\documentclass{amsart}
\usepackage[pdftex]{graphicx}

\usepackage{mathtools}
\usepackage{afterpage}
\usepackage{enumitem}
\usepackage{multirow}

\usepackage{bm}
\usepackage{color}
\usepackage{caption} 
\mathtoolsset{showonlyrefs=true}

\numberwithin{equation}{section}

\newtheorem{theorem}{Theorem}[section]
\newtheorem{lemma}[theorem]{Lemma}
\newtheorem{corollary}[theorem]{Corollary}
\newtheorem{proposition}[theorem]{Proposition}
\newtheorem{claim}[theorem]{Claim}

\theoremstyle{definition}
\newtheorem{definition}[theorem]{Definition}

\newtheorem*{acknowledgment}{Acknowledgments}

\theoremstyle{remark}

\numberwithin{equation}{section}

\newcommand{\lk}{{\rm lk}}

\begin{document}

\title{Non--braid positive hyperbolic $L$--space knots}

\author{Keisuke Himeno
}

\address{Graduate School of Advanced Science and Engineering, Hiroshima University,
1-3-1 Kagamiyama, Higashi-hiroshima, 7398526, Japan}
\email{himeno-keisuke@hiroshima-u.ac.jp}

\begin{abstract}
An $L$--space knot is a knot that admits a positive Dehn surgery yielding an $L$--space. Many known hyperbolic $L$--space knots are braid positive, meaning they can be represented as the closure of a positive braid. Recently, Baker and Kegel showed that the hyperbolic $L$--space knot $o9\_30634$ from Dunfield's census is not braid positive, and they constructed infinitely many candidates for hyperbolic $L$--space knots that may not be braid positive. However, it remains unproven whether their examples are genuinely non--braid positive. In this paper, we construct infinitely many hyperbolic $L$--space knots that are not braid positive, and our examples are distinct from those considered by Baker and Kegel.
\end{abstract}

\renewcommand{\thefootnote}{}

\maketitle


\section{Introduction}
An {\it $L$--space\/} is a rational homology $3$--sphere whose (hat version) Heegaard Floer homology has rank equal to the order of its first homology. 
A knot is called an {\it $L$--space knot\/} if it admits a positive Dehn surgery yielding an $L$--space. 
$L$--space knots were originally motivated by the study of knots admitting lens space surgeries \cite{OS05A}, and they continue to be an active subject of research.

In this paper, we consider the braid positivity of $L$--space knots. A knot or link is said to be {\it braid positive\/} if it can be expressed as the closure of a positive braid. 
Many known $L$--space knots are braid positive; for example, positive torus knots are  $L$--space knots, and then they are braid positive.
On the other hand, it is known that not all $L$--space knots are braid positive; for example, the $(2,3)$--cable of the right-handed trefoil is not braid positive (see Example 1 of \cite{ABGKLMOSTWW23}).
However, the existence of non--braid positive hyperbolic $L$--space knots remained an open problem for some time (see Problem 31.2 of \cite{HLR17}, Question 2 of \cite{ABGKLMOSTWW23}).

Recently, Baker and Kegel examined Dunfield's list of $632$ hyperbolic $L$--space knots. They showed that all but one of these knots are braid positive, and that the exceptional knot, $o9\_30634$, is not braid positive \cite{BK24}. Furthermore, they constructed infinitely many candidates for non--braid positive hyperbolic $L$--space knots, although they could not prove that any of their examples are genuinely non--braid positive.

The purpose of this paper is to give infinitely many hyperbolic $L$--space knots that are not braid positive. 
While all candidates by Baker and Kegel are represented as closures of $4$--braids, we construct the infinite family of such knots by increasing the number of strands in the braid. For each $n \ge 2$, we define a knot $K_n$ as the closure of a $2n$--braid, as follows. Let $[\varepsilon_1 i_1, \varepsilon_2 i_2,\ldots, \varepsilon_m i_m]$ denote the braid $\sigma_{i_1}^{\varepsilon_1}\sigma_{i_2}^{\varepsilon_2}\cdots\sigma_{i_m}^{\varepsilon_m}$, where $\sigma_i$ is the $i$-th standard generator of the degree $2n$ braid group and $\epsilon_j=\pm 1$. Define the braid  
\[
X_n=[n,n-1,n+1,n-2,n,n+2,\ldots,1,3,\ldots,2n-1,2n-2,\ldots,4,2,\ldots,n+1,n-1,n],
\]
as illustrated in Figure \ref{X_n}.
Let $K_n\ (n\ge 2)$ be a knot represented by the closure of $2n$--braid
\[
\beta_n=X_n^3\cdot [-1,-2,\ldots,-(n-1),n,n-1,n-2,\ldots,2,1,1,2,3,\ldots,n],
\]
see Figure \ref{K_n}. Note that $K_2$ coincides with the knot $o9\_30634$. We also regard $K_1$ as the $(2,5)$--torus knot. Although $K_n$ for odd $n$ does not appear in the statement of Theorem \ref{thm_main}, we use them in the proof.

\begin{figure}
\centering
\includegraphics[scale=0.1]{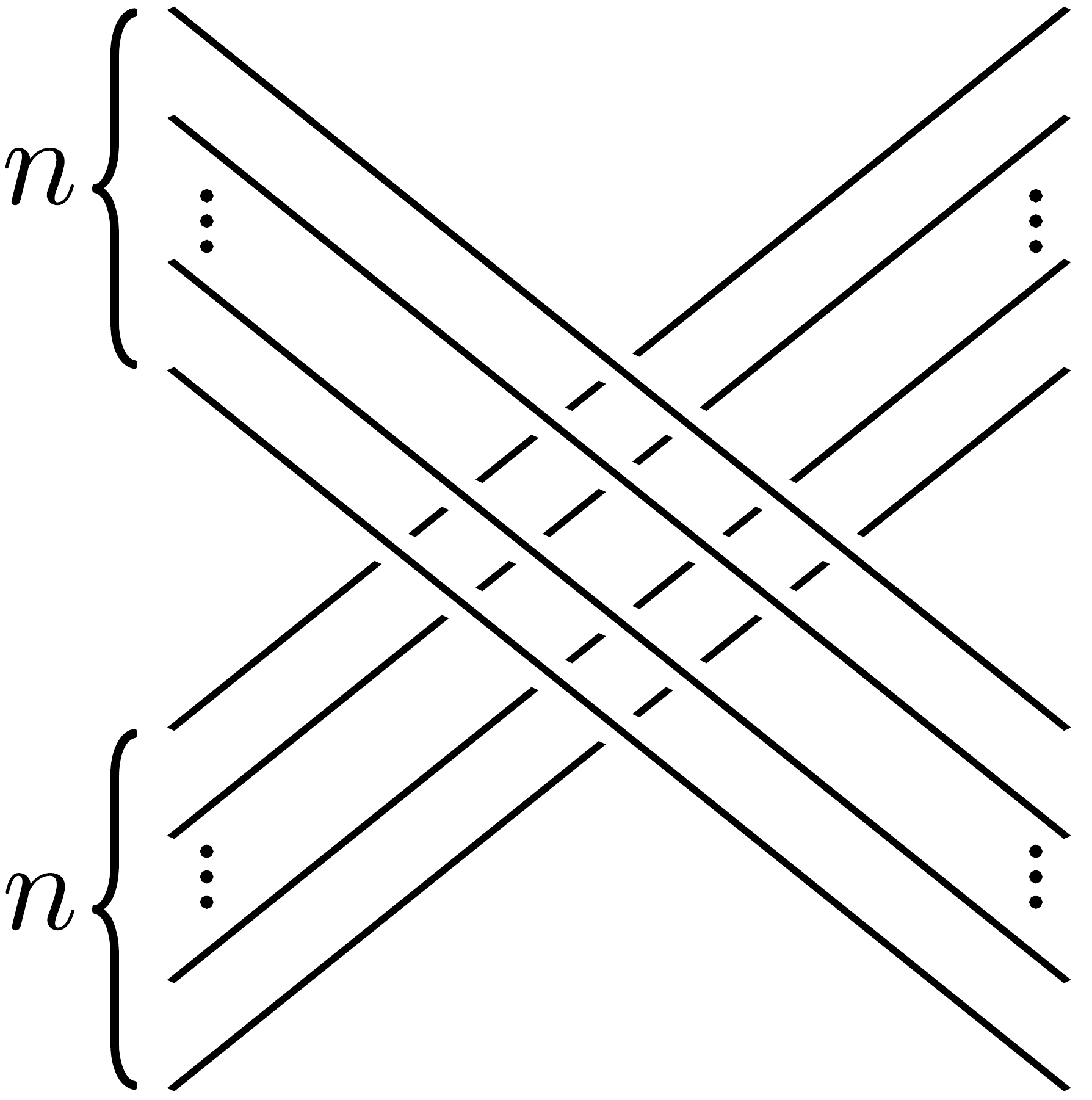}
\caption{The braid $X_n$.}
\label{X_n}
\end{figure}

\begin{figure}
\centering
\includegraphics[scale=0.13]{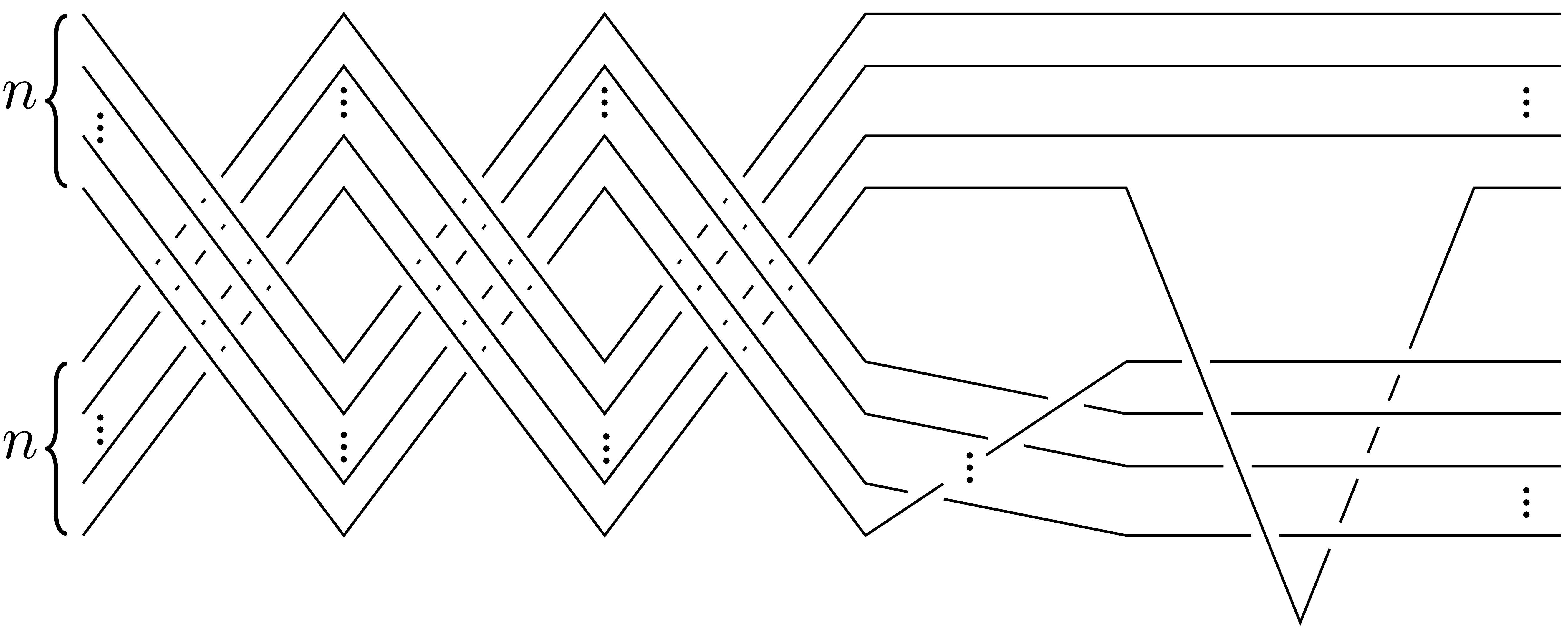}
\caption{The knot $K_n$ is the closure of this braid.}
\label{K_n}
\end{figure}

\begin{theorem}\label{thm_main}
When $n$ is even, the knot $K_n$ is a hyperbolic $L$--space knot that is not braid positive.
\end{theorem}

Since the knots $\{K_n\}$ are mutually distinct (see Lemma \ref{lem_Kn_topterm}), we have the following corollary.

\begin{corollary}
There exist infinitely many non--braid positive hyperbolic $L$--space knots.
\end{corollary}

We remark that when $n$ is odd, we also expect $K_n$ to be a non--braid positive hyperbolic $L$--space knot. However, in this case, the criterion we use for braid positivity failed to detect it.

 \begin{acknowledgment}
The author would like to thank Masakazu Teragaito for his thoughtful guidance and helpful discussions about this work. Additional thanks to Tetsuya Ito, who gave helpful comments for the HOMFLY polynomial. The author was supported by JST SPRING, Grant Number JPMJSP2132.
\end{acknowledgment}

\section{Non braid positivity}
In this section, we prove that $K_n$ is not braid positive when $n$ is even. Our proof is based on Ito's criterion using the HOMFLY polynomial \cite{Ito22B}.

\subsection{HOMFLY polynomial and its zeroth coefficient polynomial}
For an oriented link $L$, the HOMFLY polynomial $P_L(v,z)$ is a two--variable Laurent polynomial defined by the skein relation
\[
v^{-1}P_{L_+}(v,z)-vP_{L_-}(v,z)=zP_{L_0}(v,z),
\]
together with $P_{U}(v,z)=1$, where $U$ is the unknot. Here, the links (or diagrams) $L_+$, $L_-$ and $L_0$ coincide outside a small $3$--ball, and inside the $3$--ball, they differ by having the $+1$, $-1$ and $0$--tangle, respectively. Throughout this paper, all braids are assumed to be oriented from left to right.

In \cite{Ito22B}, Ito provided a criterion for the braid positivity of a link using the HOMFLY polynomial. We state the result in the case of knots. Let 
\[
\widetilde P_K(\alpha,z)=(-\alpha)^{-g(K)}P_K(v,z)|_{-v^2=\alpha},
\]
where $g(K)$ denotes the genus of a knot $K$.

\begin{theorem}[{\cite[Theorem 2]{Ito22B}}]\label{thm_Ito_braidpositive}
If $K$ is a braid positive knot, then $\widetilde P_K(\alpha,z)$ is positive, that is, all non-zero coefficients of $\widetilde P_K(\alpha,z)$ are positive integers.
\end{theorem}

The HOMFLY polynomial $P_L(v,z)$ can be expressed in the form
\[
P_L(v,z)=(v^{-1}z)^{-\# L+1}\sum_{i=0}p^i_{L}(v)z^{2i},
\]
where $\# L$ is the number of components of the link $L$. The polynomial $p^i_L(v)$ is called the {\it $i$-th coefficient (HOMFLY) polynomial\/} of $L$. In this paper, we focus on the {\it zeroth coefficient polynomial} $p^0_L(v)$. It is known that this polynomial satisfies several important properties\textup{;} see, for example \cite[Section 2]{Ito22A} and \cite[Section 2]{Tak24}.

The zeroth coefficient polynomial $p^0_L(v)$ satisfies the following skein relation\textup{:}
\begin{equation}\label{zeroth_skein}
v^{-2}p^0_{L_+}(v)-p^0_{L_-}(v)=\left\{
\begin{array}{ll}
p^0_{L_0}(v) & (\delta=0),\\
0 & (\delta=1),
\end{array}
\right.
\end{equation}
where $\delta=\frac{1}{2}(\#L_+-\#L_0+1)\in\{0,1\}$. In particular, $\delta=0$ if $L_+$ is a knot. 
Furthermore, for a two--component link $L=k_1\cup k_2$, we have
\begin{equation}\label{zeroth_skein_link}
p^0_L(v)=(v^{-2}-1)v^{2\cdot \lk(k_1,k_2)}p^0_{k_1}(v)p^0_{k_2}(v),
\end{equation}
where $\lk(k_1,k_2)$ is the linking number of $k_1$ and $k_2$. 
In particular, when $L_+$ is a knot (hence $L_-$ is also a knot and $L_0$ is a two--component link $k_1\cup k_2$), the skein relation \eqref{zeroth_skein} can be rewritten as
\begin{equation}\label{zeroth_skein_NtoP}
p^0_{L_-}(v)=v^{-2}p^0_{L_+}(v)+(1-v^{-2})v^{2\cdot \lk(k_1,k_2)}p^0_{k_1}(v)p^0_{k_2}(v).
\end{equation}


\subsection{Degree of $p^0_L$ for a braid positive link $L$}
For a braid $\beta$, let $L(\beta)$ denote the closure of $\beta$. 
A positive braid $\beta$ is a {\it minimal positive braid} if the number of strands of $\beta$ is minimum among all the positive braid representative of $L(\beta)$.
The following proposition is useful when applying the skein relation to a braid positive link.

\begin{proposition}[{\cite[Lemma 2]{vB85}}]\label{Prop_positive_fulltwist}
Let $L$ be a braid positive link that is not an unlink. Then there exists a positive braid $\beta$ such that $L$ is the closure of a positive braid of the form $\sigma_i^2\beta$ for some $i$. Moreover, such a positive braid representative $\sigma_i^2\beta$ of $L$ can be taken so that it is a minimal positive braid.
\end{proposition}

\begin{lemma}\label{Lem_zeroth_degree_bound}
Let $\beta$ be a positive $n$--braid, and let $e$ be the number of crossings of $\beta$. 
Then the degree of the zeroth coefficient polynomial satisfies
\[
\deg p^0_{L(\beta)}(v)\le n+e-\#L(\beta).
\]
\end{lemma}
\begin{proof}
We prove this lemma by induction on $e$. For an unlink $U$, 
\[
p^0_U(v)=(v^{-2}-1)^{\#U-1},
\] 
so $\deg p^0_U(v)=0$.

Assume that $L(\beta)$ is not an unlink. By Proposition \ref{Prop_positive_fulltwist}, we can write $L(\beta)=L(\sigma_i^2 \beta')$ for some $i$, where $\beta'$ is a positive $n$--braid with $e-2$ crossings. (If necessary, we may increase the number of crossings to realize $\beta'$ as an $n$--braid.)

Applying the skein relation \eqref{zeroth_skein}, we get
\[
p^0_{L(\beta)}(v)=p^0_{L(\sigma_i^2\beta')}(v)=v^2p^0_{L(\beta')}(v)+\left\{
\begin{array}{ll}
v^2p^0_{L(\sigma_i\beta')}(v) & (\delta=0),\\
0 & (\delta=1).
\end{array}
\right.
\]
Note that $\#L(\beta')=\#L(\beta)$, and $\#L(\sigma_i\beta')=\#L(\beta)+1$ if $\delta=0$.
By the assumption of induction, we have
\begin{align*}
\deg v^2p^0_{L(\beta')}(v)&\le 2+(n+e-2-\#L(\beta'))\\
&=n+e-\#L(\beta),
\end{align*} 
and if $\delta=0$,
\begin{align*}
\deg v^2p^0_{L(\sigma_i\beta')}(v)&\le 2+(n+e-1-\#L(\sigma_i\beta'))\\
&=n+e-\#L(\beta).
\end{align*}
Therefore, the claim follows.
\end{proof}

\begin{definition}
A positive $n$--braid $\beta$  is said to be {\it sharp\/} if 
\[
\deg p^0_L(\beta)(v)=n+e-\#L(\beta),
\]
where $e$ is the number of crossings of $\beta$.
\end{definition}

We now state several lemmas concerning the notion of sharpness.

\begin{lemma}\label{Lem_nonsharp_1}
If positive braid $\sigma_i^2\beta$ is sharp, then at least one of $\beta$ and $\sigma_i\beta$ is sharp.
\end{lemma}
\begin{proof}
The claim follows immediately from the proof of Lemma \ref{Lem_zeroth_degree_bound}.
\end{proof}

\begin{lemma}\label{Lem_nonsharp_2}
If a positive braid $\beta$ is not a minimal positive braid, then $\beta$ is not sharp.
\end{lemma}
\begin{proof}
Let $\beta$ be a positive $n$--braid that is not minimal, and let $\beta'$ be a minimal positive braid representative of $L(\beta)$ with $n'$ strands ($n'<n$). The number of crossings of $\beta$ and $\beta'$ are given by
\[
n-\chi(L(\beta))\ \text{and}\ n'-\chi(L(\beta)),
\]
respectively \cite{Sta76}. Here $\chi(L)$ is the maximal Euler characteristic among all compact, connected, oriented surfaces whose boundary is a link $L$.

By Lemma \ref{Lem_zeroth_degree_bound}, 
\begin{align*}
\deg p^0_{L(\beta)}(v) &\le n'+n'-\chi(L(\beta))-\#L(\beta)\\
&<n+n-\chi(L(\beta))-\#L(\beta).
\end{align*}
This implies that $\beta$ is not sharp.
\end{proof}

\begin{lemma}\label{Lem_nonsharp_3}
Let $L=k_1 \cup k_2$ be a two--component link represented as the closure of a positive braid $\beta$. Suppose that each component knot $k_1$ and $k_2$ can also be represented as the closure of positive braids $\beta_1$ and $\beta_2$, respectively.
If $\beta$ is sharp, then both $\beta_1$ and $\beta_2$ are sharp.
\end{lemma}
\begin{proof}
Let $n,n_1,n_2$ be the numbers of strands, and $e,e_1,e_2$ be the numbers of crossings of $\beta$, $\beta_1$ and $\beta_2$, respectively. Then, 
\[
n=n_1+n_2,\ e=e_1+e_2+\lk(k_1,k_2).
\]
By \eqref{zeroth_skein_link} and Lemma \ref{Lem_zeroth_degree_bound}, we have
\begin{align*}
\deg p_L^0 &\le 2\cdot \lk(k_1,k_2)+(n_1+e_1-1)+(n_2+e_2-1)\\
&=n+e-2.
\end{align*}
The equality holds if and only if $\beta_1$ and $\beta_2$ are sharp.
Since $\beta$ is sharp, the equality holds, and the result follows. 
\end{proof}

\subsection{Non--braid positivity for $K_n$}
We are ready to prove that $K_n$ is not braid positive when $n$ is even.

\begin{lemma}\label{lem_Kn_topterm}
For $n\ge 2$, the top term of $p^0_{K_n}(v)$ is $(-1)^nv^{3n^2+3n}$.
\end{lemma}
\begin{proof}
We prove this by induction on $n$. Using the Sage Mathematics Software System \cite{Sage}, one can confirm that the top term of $p^0_{K_2}(v)$ is $v^{18}$.

Throughout the following, we identify each braid with its closure diagram in the skein tree.

\begin{figure}
\centering
\includegraphics[scale=0.07]{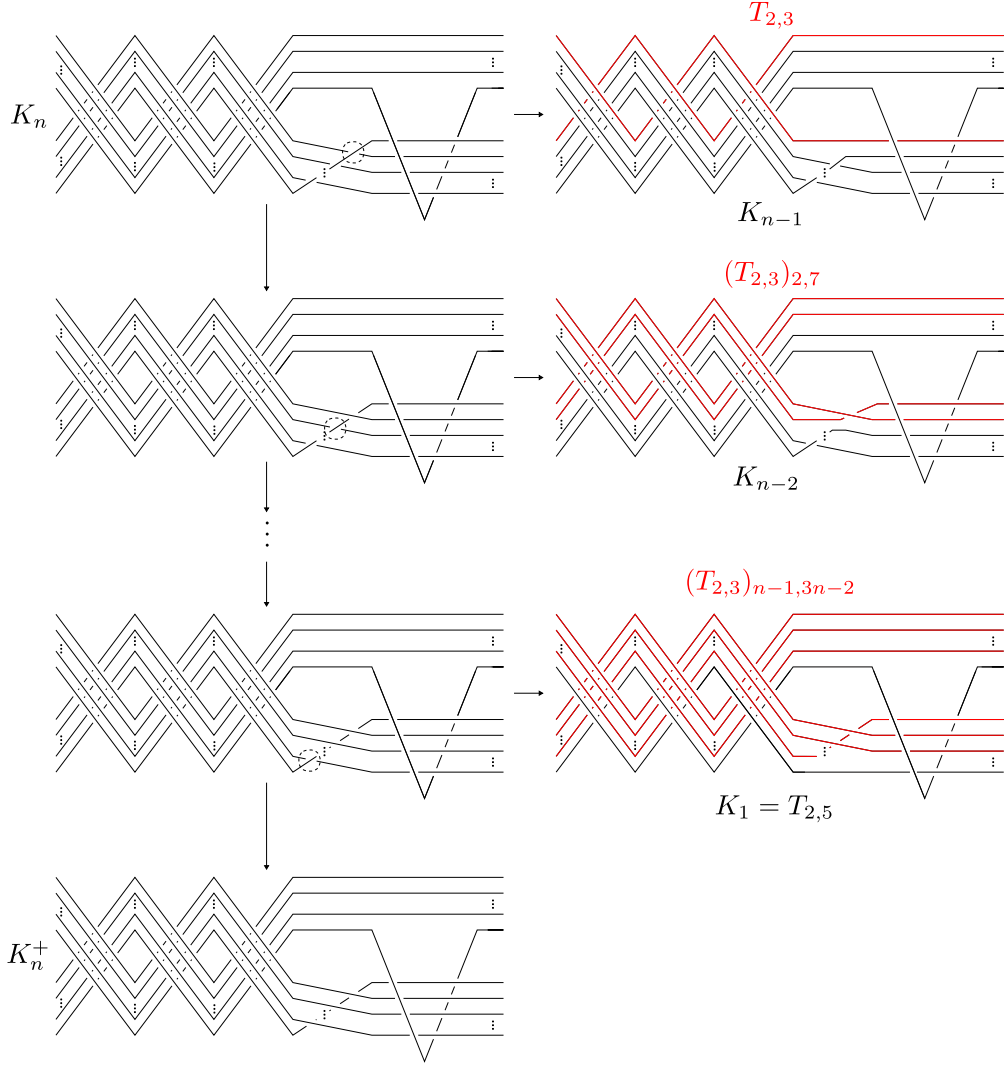}
\caption{The skein tree for the knot $K_n$. Consider the crossing change and the smoothing at the crossing indicated by the dashed circle.}
\label{K_n_main_skein}
\end{figure}

Assume that $n\ge 3$. From the skein tree shown in Figure \ref{K_n_main_skein} and equation \eqref{zeroth_skein_NtoP}, we obtain
\begin{align*}
p^0_{K_n}(v) =& (1-v^{-2})v^{2(3(n-1)\cdot 1+1)}p^0_{T_{2,3}}(v)p^0_{K_{n-1}}(v)\\
&+(1-v^{-2})v^{-2}\cdot v^{2(3(n-2)\cdot 2+2)}p^0_{(T_{2,3})_{2,7}}(v)p^0_{K_{n-2}}(v)\\
&+\cdots\\
&+(1-v^{-2})v^{-2(n-2)}\cdot v^{2(3\cdot1\cdot(n-1)+n-1}p^0_{(T_{2,3})_{n-1,3n-2}}(v)p^0_{K_1}(v)\\
&+v^{-2(n-1)}p_{K_n^+}(v)\\
=& \sum_{k=1}^{n-1}v^{-2(k-1)}(1-v^{-2})v^{2(3(n-k)\cdot k+k)}p^0_{(T_{2,3})_{k,3k+1}}(v)p^0_{K_{n-k}}(v)\\
&+v^{-2(n-1)}p_{K_n^+}(v),
\end{align*}
where $(T_{2,3})_{k,3k-1}$ denotes the $(k,3k-1)$--cable of the right-handed trefoil, $K_n^+$ is the braid positive knot shown in the bottom of Figure \ref{K_n_main_skein}.

By the induction assumption and the fact $\deg p^0_{K_1}(v)=\deg p^0_{T_{2,5}}(v)=6$, we have $\deg p^0_{K_i}=3i^2+3i$ for $1\le i\le n-1$. Since $(T_{2,3})_{k,3k+1}$ is represented by the closure of the positive braid $X_k^3\cdot[1,2,\ldots,k-1]$, Lemma \ref{Lem_zeroth_degree_bound} gives
\begin{align*}
&\deg \left(v^{-2(k-1)}(1-v^{-2})v^{2(3(n-k)\cdot k+k)}p^0_{(T_{2,3})_{k,3k+1}}(v)p^0_{K_{n-k}}(v) \right)\\
&\le -2(k-1)+2(3(n-k) k+k)+(2k+3k^2+(k-1)-1)+3(n-k)^2+3(n-k)\\
&=3n^2+3n.
\end{align*}
Equality holds if and only if the positive braid $X_k^3\cdot[1,2,\ldots,k-1]$ is sharp.

\begin{claim}\label{claim_cable_nonsharp}
For $k\ge 2$, the positive braid $X_k^3\cdot[1,2,\ldots,k-1]$ is not sharp. 
\end{claim}
\begin{proof}
We consider the skein tree
\[
\begin{array}{ccccccccccccc}
D_0 & \rightarrow & D_1 & \rightarrow & \cdots & \rightarrow & D_i & \rightarrow & D_{i+1} & \rightarrow & \cdots & \rightarrow & D_k,\\
\downarrow & & \downarrow & & &  & \downarrow & & \downarrow & & \\
D_0^s & & D_1^s & & & & D_i^s & &D_{i+1}^s
\end{array}
\]
where $D_0$ is the braid $X_k^3 \cdot [1,2,\ldots,k-1]$ and each $D_i\ (i=1,\ldots,k)$ and $D_i^s\ (i=0,\ldots,k-1)$ are as illustrated in Figures \ref{cable_skein_i} and \ref{cable_skein_i_smooth}. The right arrows correspond to crossing changes, and the down arrows correspond to smoothings. Note that all links in the skein tree are braid positive links.

$D_k$ contains exactly one occurrence of the generator $\sigma_{2k-1}$, and is hence a non-minimal positive braid. Each $D_i^s$ ($0 \le i \le k-1$) is also non-minimal as shown in Figures \ref{cable_skein_i_smooth_nonsharp} and \ref{cable_skein_k-1_smooth_nonsharp}. Thus, by Lemma \ref{Lem_nonsharp_1} and \ref{Lem_nonsharp_2}, the claim follows.
\end{proof}

\begin{figure}
\centering
\includegraphics[scale=0.12]{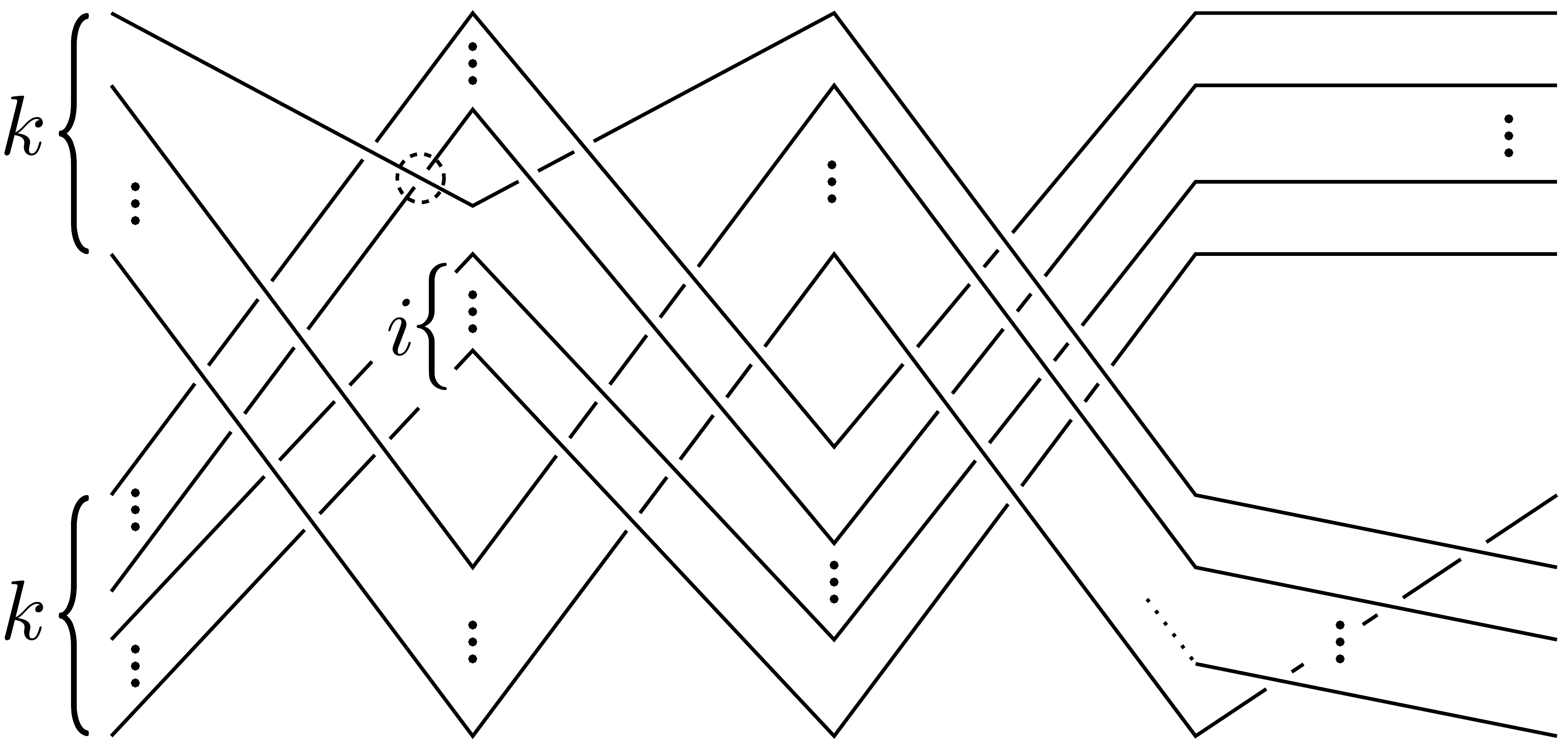}
\caption{The diagram $D_i\ (0\le i\le k)$. Performing a crossing change at the crossing indicated by the dashed circle, followed by a Reidemeister II move, results in the diagram $D_{i+1}$. Alternatively, smoothing the crossing produces the diagram $D_i^s$.}\label{cable_skein_i}
\end{figure}

\begin{figure}
\centering
\includegraphics[scale=0.12]{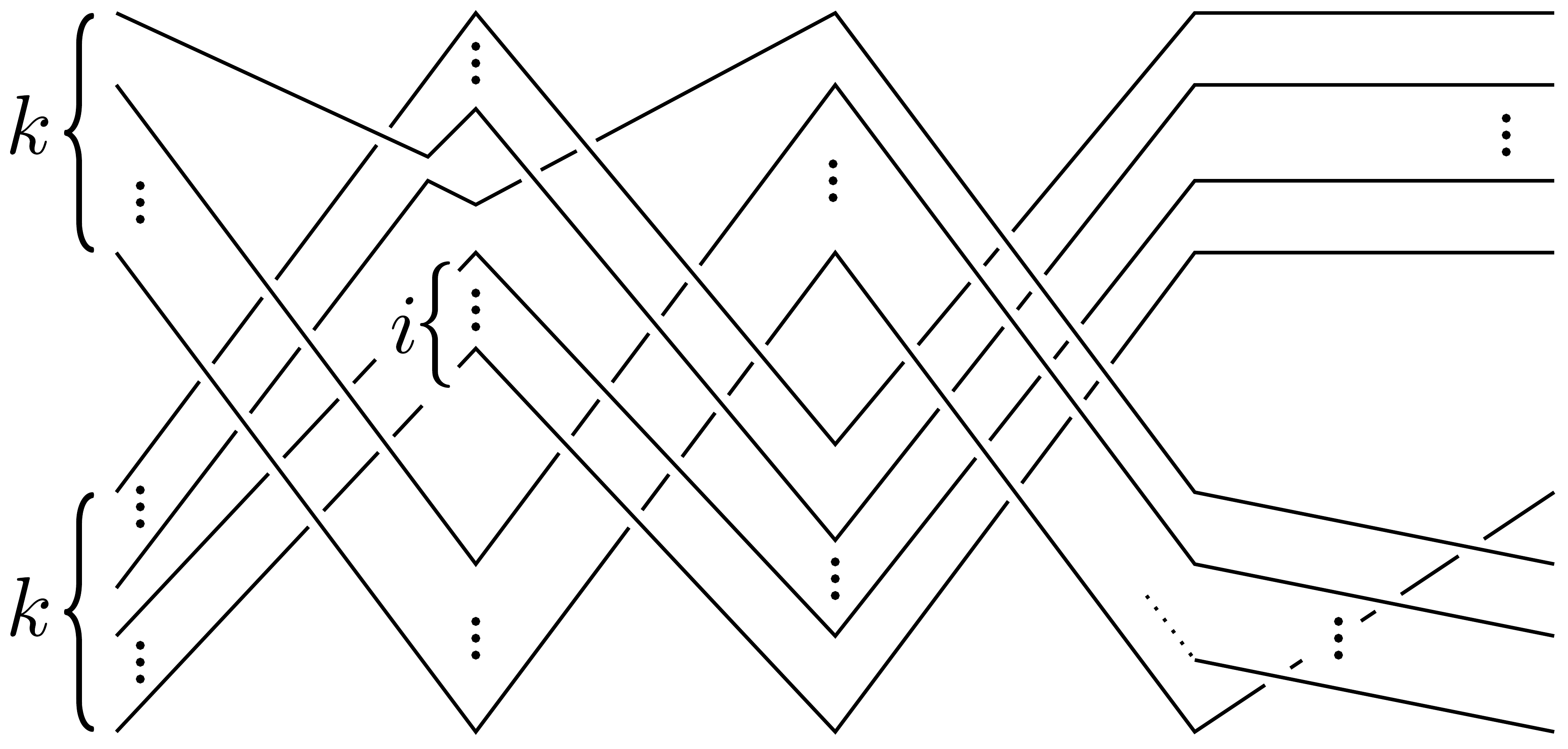}
\caption{The diagram $D_i^s\ (0\le i\le k-1)$.}\label{cable_skein_i_smooth}
\end{figure}

\begin{figure}
\centering
\includegraphics[scale=0.1]{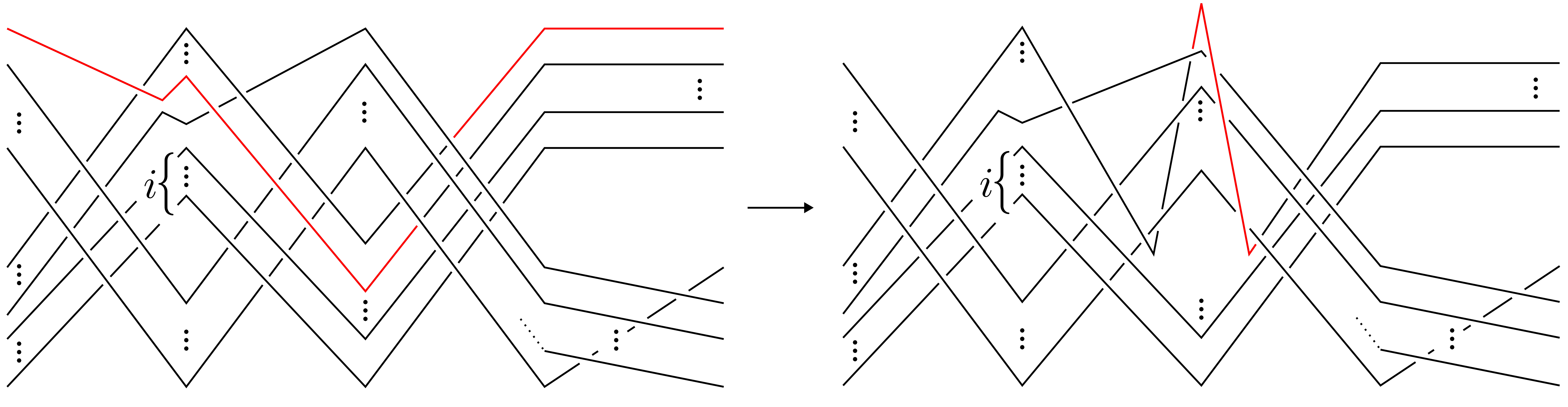}
\caption{The positive braid $D^s_i\ (0\le i \le k-2)$ is not minimal.}
\label{cable_skein_i_smooth_nonsharp}
\end{figure}

\begin{figure}
\centering
\includegraphics[scale=0.1]{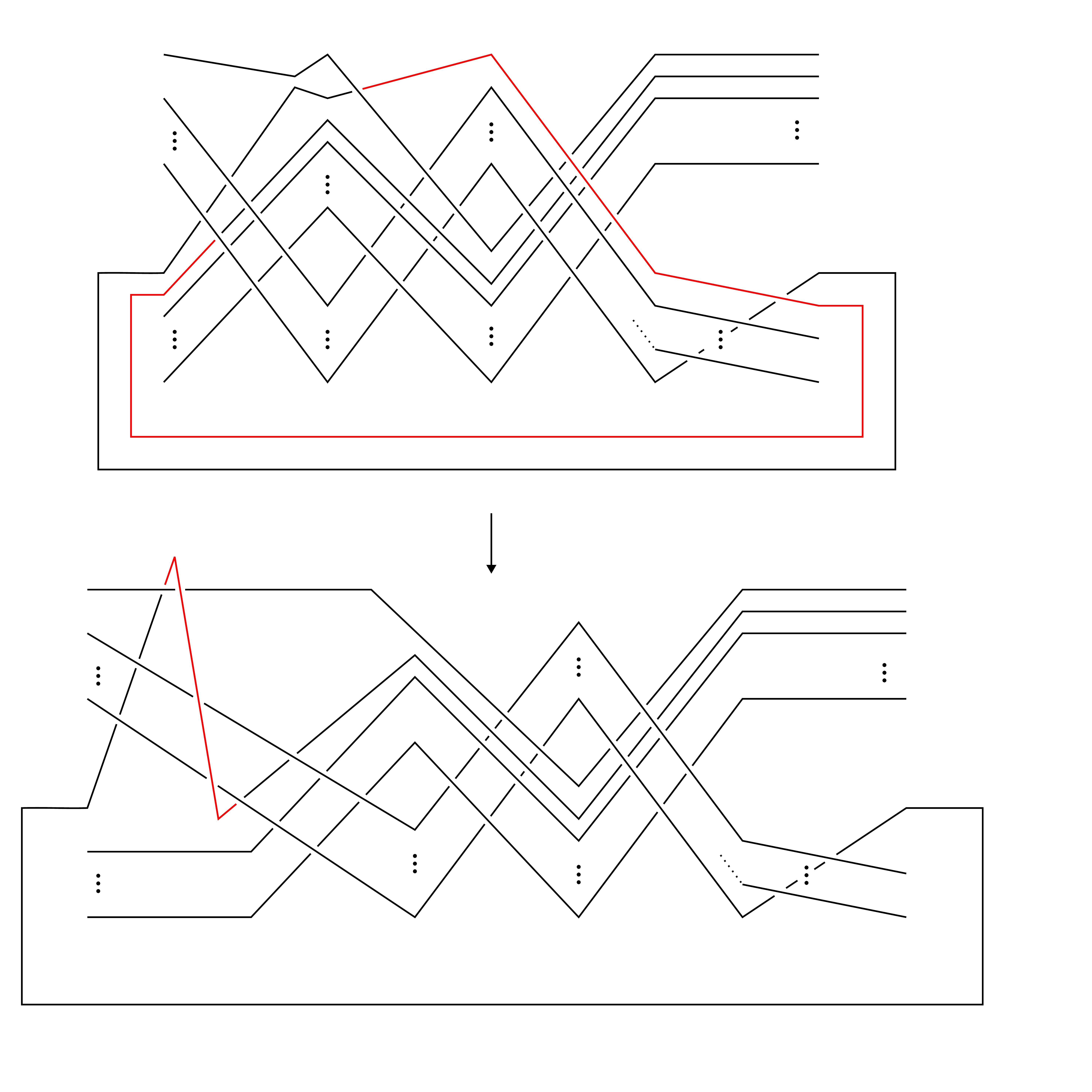}
\caption{The positive braid $D^s_{k-1}$ is not minimal.}
\label{cable_skein_k-1_smooth_nonsharp}
\end{figure}


On the other hand, $K_n^+$ is represented by the closure of the positive braid
\[
X_n^3\cdot [1,2,\ldots,n-1,n,n-1,n-2,\ldots,2,1,1,2,3,\ldots,n].
\]
Lemma \ref{Lem_zeroth_degree_bound} gives 
\begin{align*}
\deg v^{-2(n-1)}p^0_{K_n^+}(v)&\le -2(n-1)+2n+(3n^2+n-1+2n)-1\\
&=3n^2+3n.
\end{align*}
Equality holds if and only if the positive braid 
\[
X_n^3\cdot [1,2,\ldots,n-1,n,n-1,n-2,\ldots,2,1,1,2,3,\ldots,n]
\]
is sharp again.

\begin{claim}\label{claim_K+_nonsharp}
For $n\ge 3$, the positive braid
\[
X_n^3\cdot [1,2,\ldots,n-1,n,n-1,n-2,\ldots,2,1,1,2,3,\ldots,n]
\]
is not sharp.
\end{claim}

\begin{proof}
We consider the skein tree
\[
\begin{array}{ccccccccccccc}
E_0 & \rightarrow & E_1 & \rightarrow & \cdots & \rightarrow & E_i & \rightarrow & E_{i+1} & \rightarrow & \cdots & \rightarrow & E_n,\\
\downarrow & & \downarrow & & &  & \downarrow & & \downarrow & & \\
E_0^s & & E_1^s & & & & E_i^s & &E_{i+1}^s
\end{array}
\]
where $E_0$ is the braid described in the claim, and $E_n$ is $X_n^3 \cdot [1,2,\ldots,n-1]$, which is not sharp by Claim \ref{claim_cable_nonsharp}. The diagrams $E_i$ and $E_i^s$ are illustrated in Figures \ref{K_n_plus_Ei} and \ref{K_n_plus_Es}.

\begin{figure}
\centering
\includegraphics[scale=0.12]{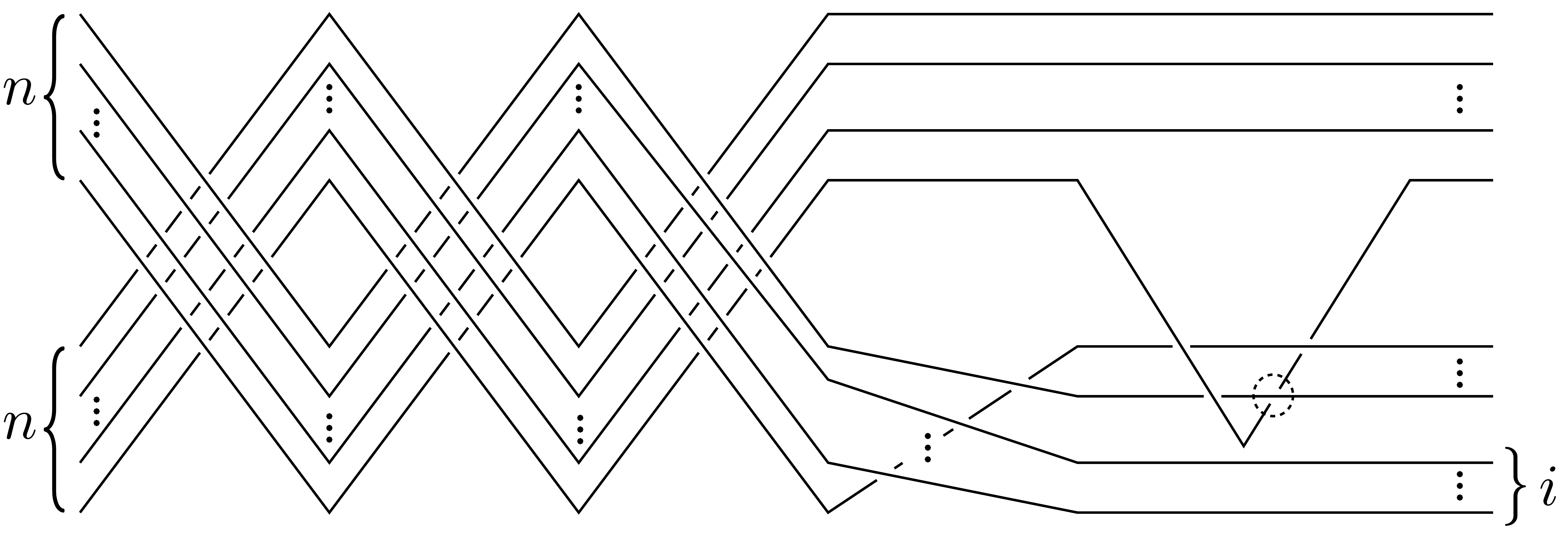}
\caption{The diagram $E_i\ (0\le i\le n)$. Performing a crossing change at the crossing indicated by the dashed circle, results in the diagram $E_{i+1}$. Alternatively, smoothing the crossing produces the diagram $E_i^s$.}\label{K_n_plus_Ei}
\end{figure}

\begin{figure}
\centering
\includegraphics[scale=0.1]{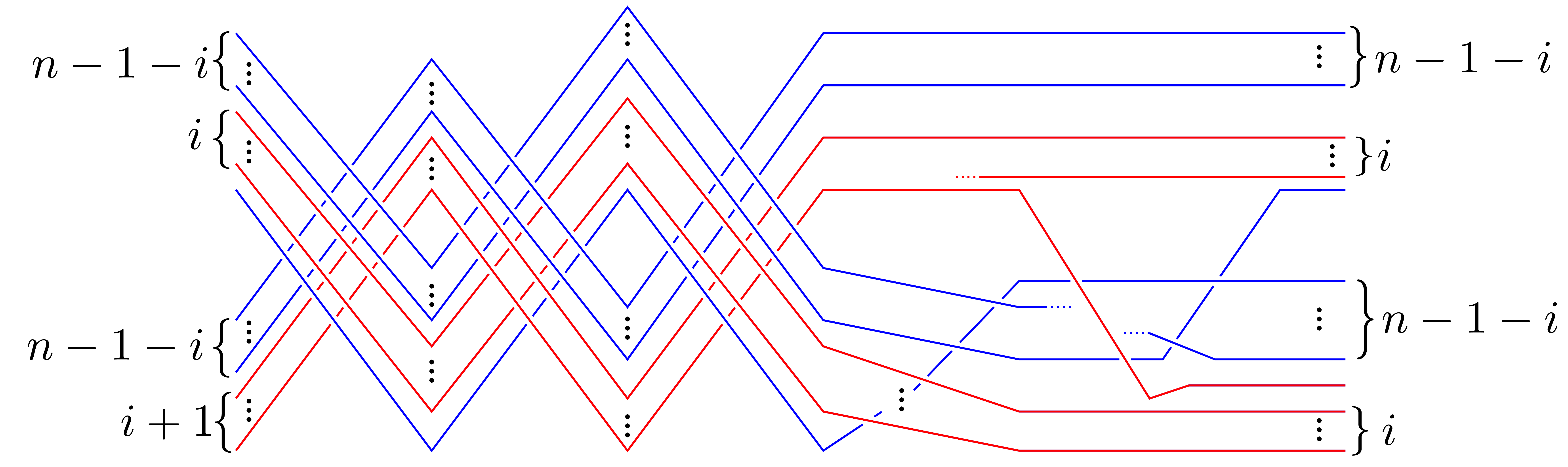}
\caption{The diagram $E_i^s\ (0\le i\le n-1)$ represents the two--component link $k_1\cup k_2$. $k_1$ is represented by the positive braid $\gamma^1_i$ (red), and $k_2$ by $\gamma^2_i$ (blue).}
\label{K_n_plus_Es}
\end{figure}

As shown in Figure \ref{K_n_plus_Es}, the diagram $E_i^s$ consists of two components $k_1$ and $k_2$, represented by the positive braids $\gamma^1_i$ and $\gamma^2_i$ respectively. Figure \ref{gamma1} shows that $\gamma^1_i$ $(1\le i\le n-1)$ is not minimal, and hence non-sharp by Lemma \ref{Lem_nonsharp_2}. Then, Lemma \ref{Lem_nonsharp_3} implies that $E_i^s$ is not sharp for $1\le i \le n-1$.

\begin{figure}
\centering
\includegraphics[scale=0.1]{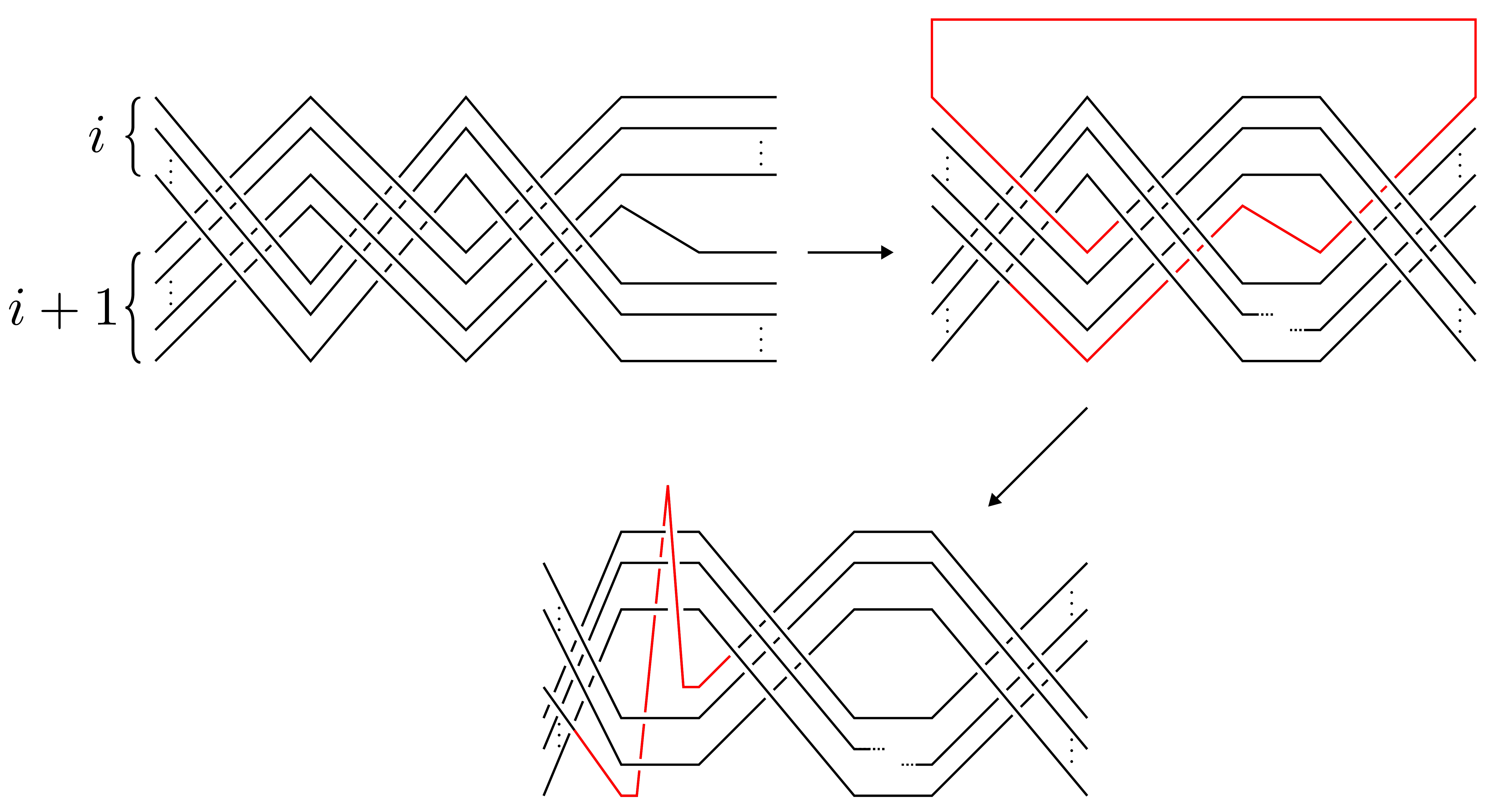}
\caption{The positive braid $\gamma^1_i\ (1\le i\le n-1)$ is not minimal. The first deformation is a conjugation that moves the ``X-shaped" part on the left side of the braid to the right.}
\label{gamma1}
\end{figure}

For $i = 0$, we observe that $\gamma^1_0$ is a $1$--braid (thus sharp), so we analyze $\gamma^2_0$. We perform a crossing change and smoothing at the crossing marked in Figure \ref{gamma2}. The resulting diagrams (Figures \ref{gamma2_c} and \ref{gamma2_s}) demonstrate that $\gamma^2_0$ is not sharp, and hence $E_0^s$ is not sharp.
\end{proof}

\begin{figure}
\centering
\includegraphics[scale=0.13]{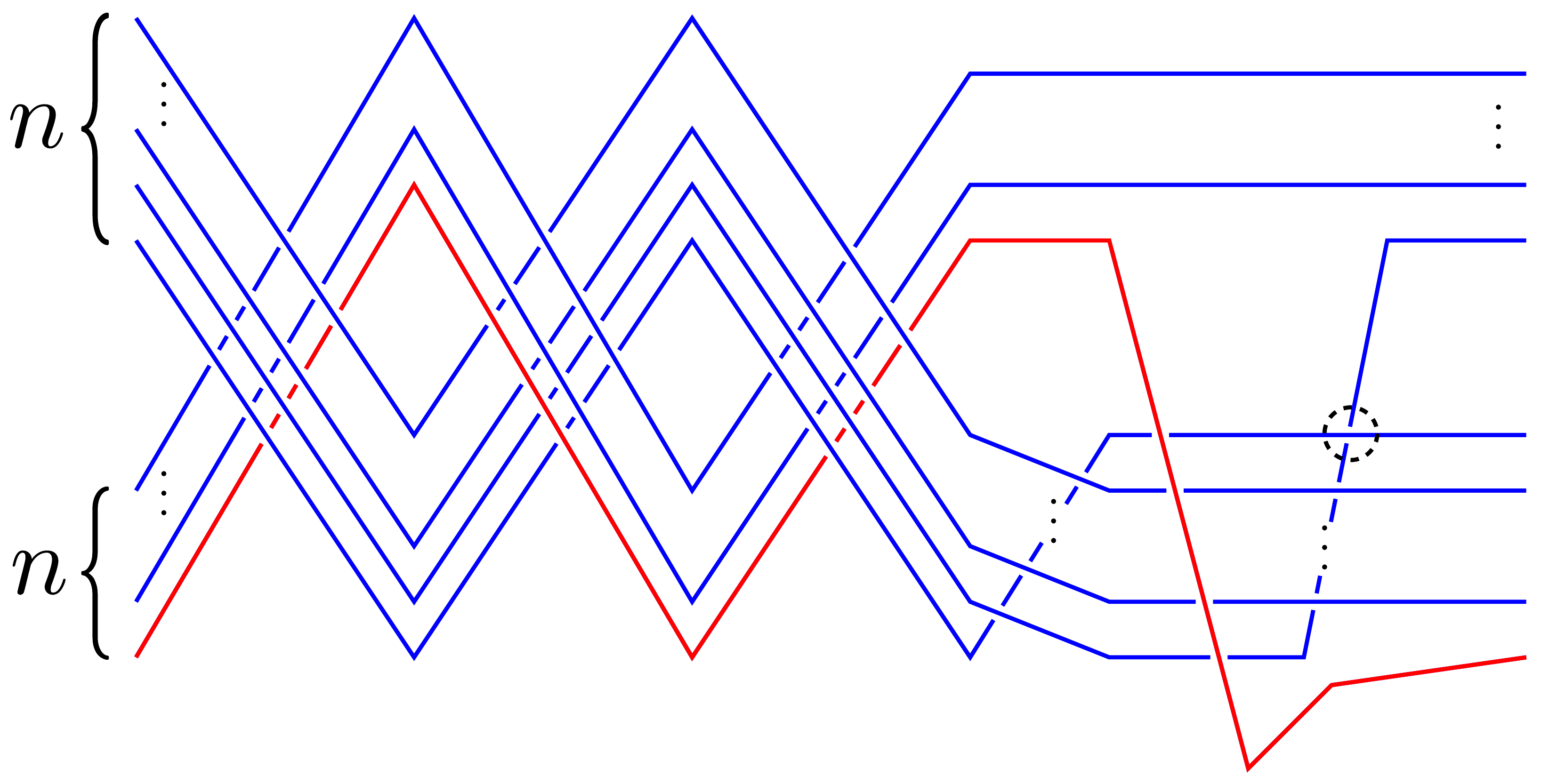}
\caption{The diagram $E_0^s$. The red line represents $\gamma^1_0$, and the blue line represents $\gamma^2_0$. The crossing change at the crossing indicated by the dashed circle changes $\gamma^2_0$ into the diagram as in Figure \ref{gamma2_c}. Alternatively, the smoothing yields the diagram as in Figure \ref{gamma2_s}.}
\label{gamma2}
\end{figure}

\begin{figure}
\centering
\includegraphics[scale=0.1]{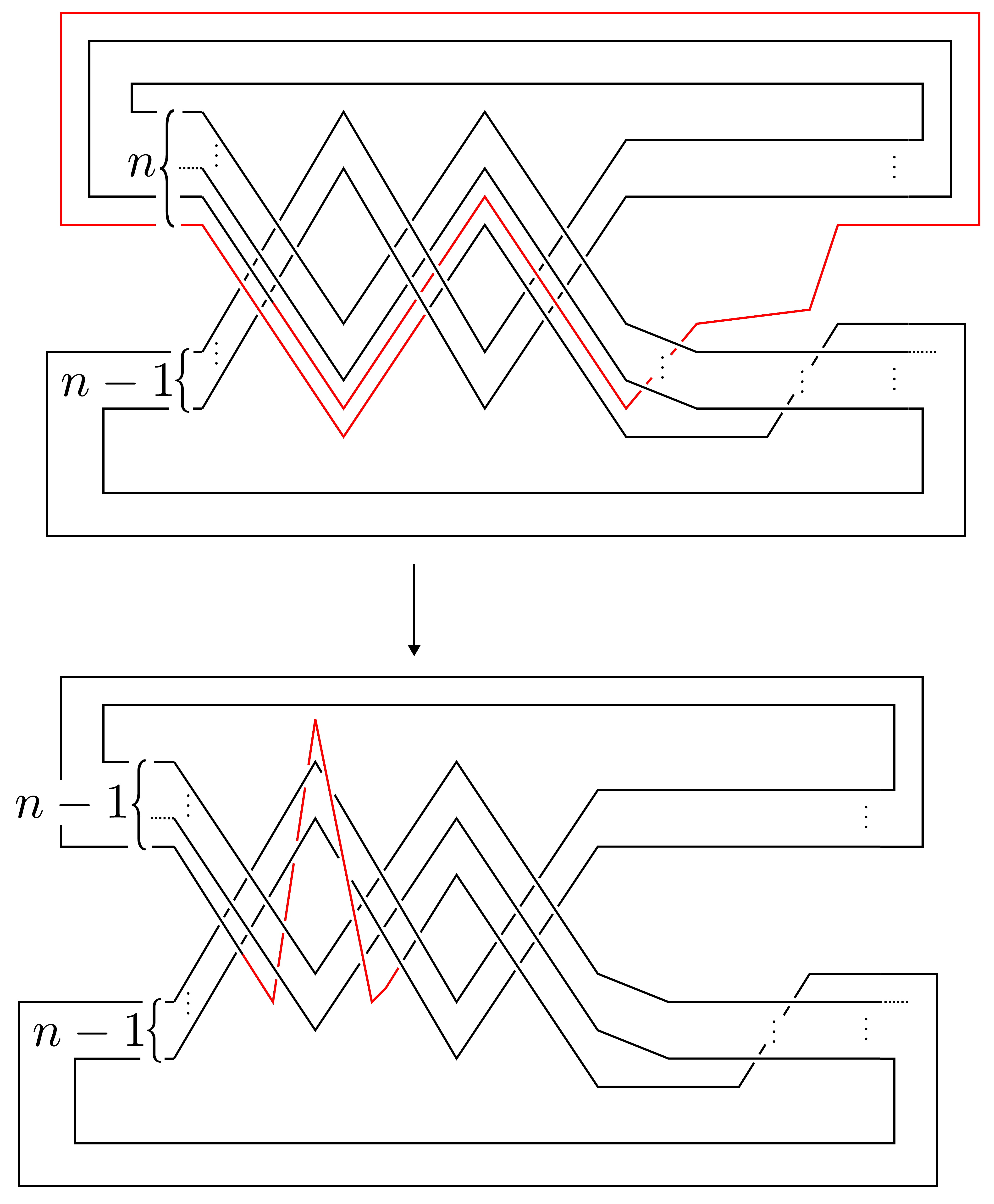}
\caption{The diagram obtained by the crossing change of $\gamma^2_0$. As shown, the corresponding positive braid is not minimal, and hence it is not sharp.}
\label{gamma2_c}
\end{figure}

\begin{figure}
\centering
\includegraphics[scale=0.13]{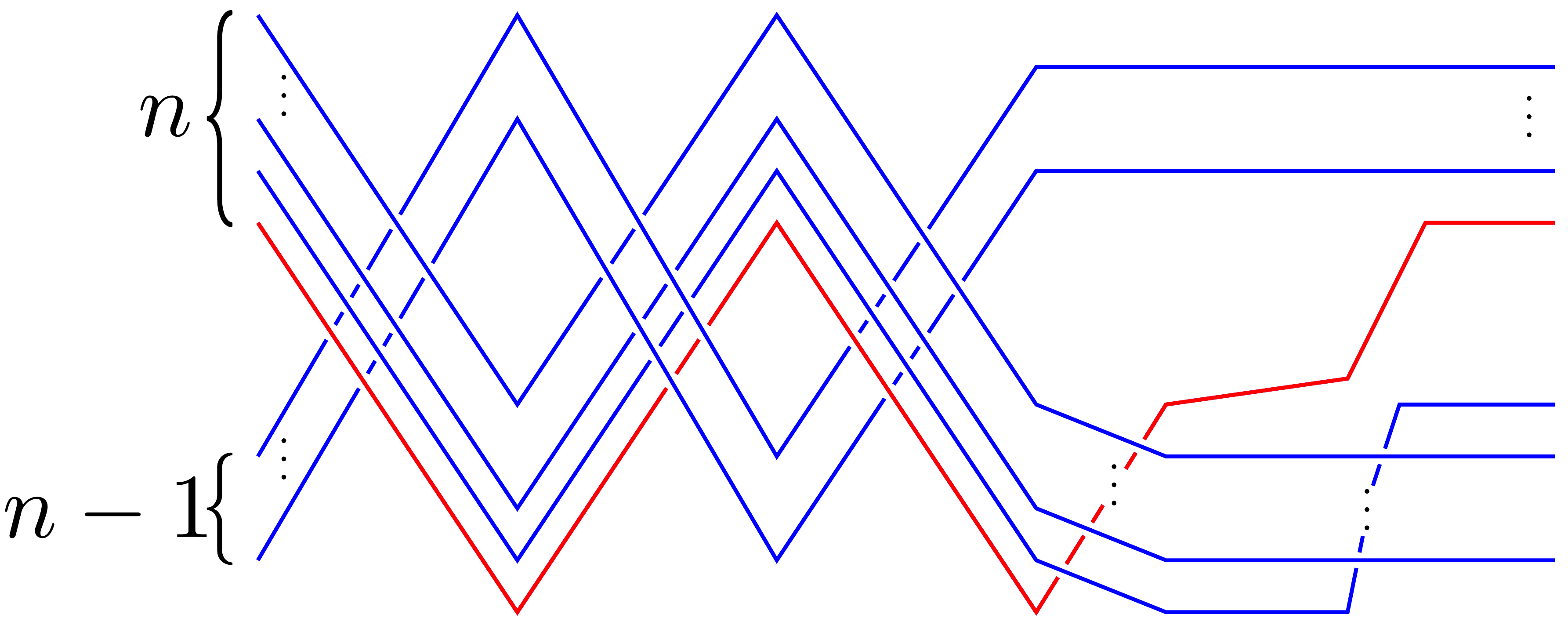}
\caption{The diagram obtained by the smoothing of $\gamma^2_0$. The blue line is the positive braid $X^3_{n-1}[1,2,\ldots,n-2]$, which is not sharp by Claim \ref{claim_cable_nonsharp}.}
\label{gamma2_s}
\end{figure}

By Claims \ref{claim_cable_nonsharp} and \ref{claim_K+_nonsharp}, the top term of $p^0_{K_n}(v)$ arises from the term 
\[
(1-v^{-2})v^{2(3(n-1)\cdot 1+1)}p^0_{T_{2,3}}(v)p^0_{K_{n-1}}(v),
\]
which evaluates to $v^{2(3(n-1)\cdot 1+1)}(-v^4)(-1)^{n-1}v^{3(n-1)^2+3(n-1)}=(-1)^nv^{3n^2+3n}$.
The proof of Lemma \ref{lem_Kn_topterm} is complete.
\end{proof}

\begin{lemma}\label{lem_genus}
When $n$ is even, the genus of $K_n$ is given by
\[
g(K_n)=\frac{3}{2}n^2-\frac{1}{2}n+1.
\]
\end{lemma}
\begin{proof}
By \cite{Ni07} and Proposition \ref{prop_Lspace} below, $K_n$ is an $L$--space knot, and hence fibered. A quasi-positive Seifert surface for $K_n$ can be constructed by attaching $3n^2+n+1$ bands to $2n$ disks, where one of the bands corresponds to the braid 
\[
[-1,-2,\ldots,-(n-1),n,n-1,\ldots,2,1]
\] 
and the remaining bands correspond to the other positive crossings. By \cite{Rud89}, the surface is incompressible and therefore serves as the fiber surface of $K_n$. A direct calculation then yields the genus $g(K_n)=\frac{3}{2}n^2-\frac{1}{2}n+1$.
\end{proof}

\begin{proposition}\label{prop_non_braid_positive}
If $n$ is even, then $K_n$ is not braid positive.
\end{proposition}
\begin{proof}
Set $n=2k$. By Lemmas \ref{lem_Kn_topterm} and \ref{lem_genus}, the top term of $(-\alpha)^{-g(K_{2k})}p^0_{K_{2k}}(v)|_{-v^2=\alpha}$ is
\begin{align*}
(-\alpha)^{-(6k^2-k+1)}(-1)^{2k}(-\alpha)^{6k^2+3k}
&=(-\alpha)^{4k-1}\\
&=-\alpha^{4k-1}.
\end{align*}
Since this is the term of $z^0$ in $\widetilde{P}_{K_n}(\alpha,z)$, $\widetilde{P}_{K_n}(\alpha,z)$ is a non-positive polynomial. Therefore, Theorem \ref{thm_Ito_braidpositive} implies that $K_n$ is not braid positive.
\end{proof}

\section{$L$--space surgery}

In this section, we prove that $K_n$ admits a Dehn surgery yielding an $L$--space if $n$ is even. Throughout, we set $n = 2k$.  
We apply the Montesinos trick \cite{Mon75}: for a strongly invertible link $L$ in $S^3$, the manifold obtained by Dehn surgery on $L$ is the double branched cover of a link $\ell$ arising from a tangle replacement on the axis. See also \cite{BK24, Ter22} for details.

Figure \ref{strongly_invertible} illustrates a strongly invertible position of the link $K \cup C_1 \cup C_2 \cup C_3$, where performing $(-1)$--surgery on $C_1$ and $1$--surgery on both $C_2$ and $C_3$ transforms $K$ into $K_{2k}$. Note that $r$--surgery on $K$ corresponds to $(8k^2 + r)$--surgery on $K_{2k}$.

\begin{figure}
\centering
\includegraphics[scale=0.14]{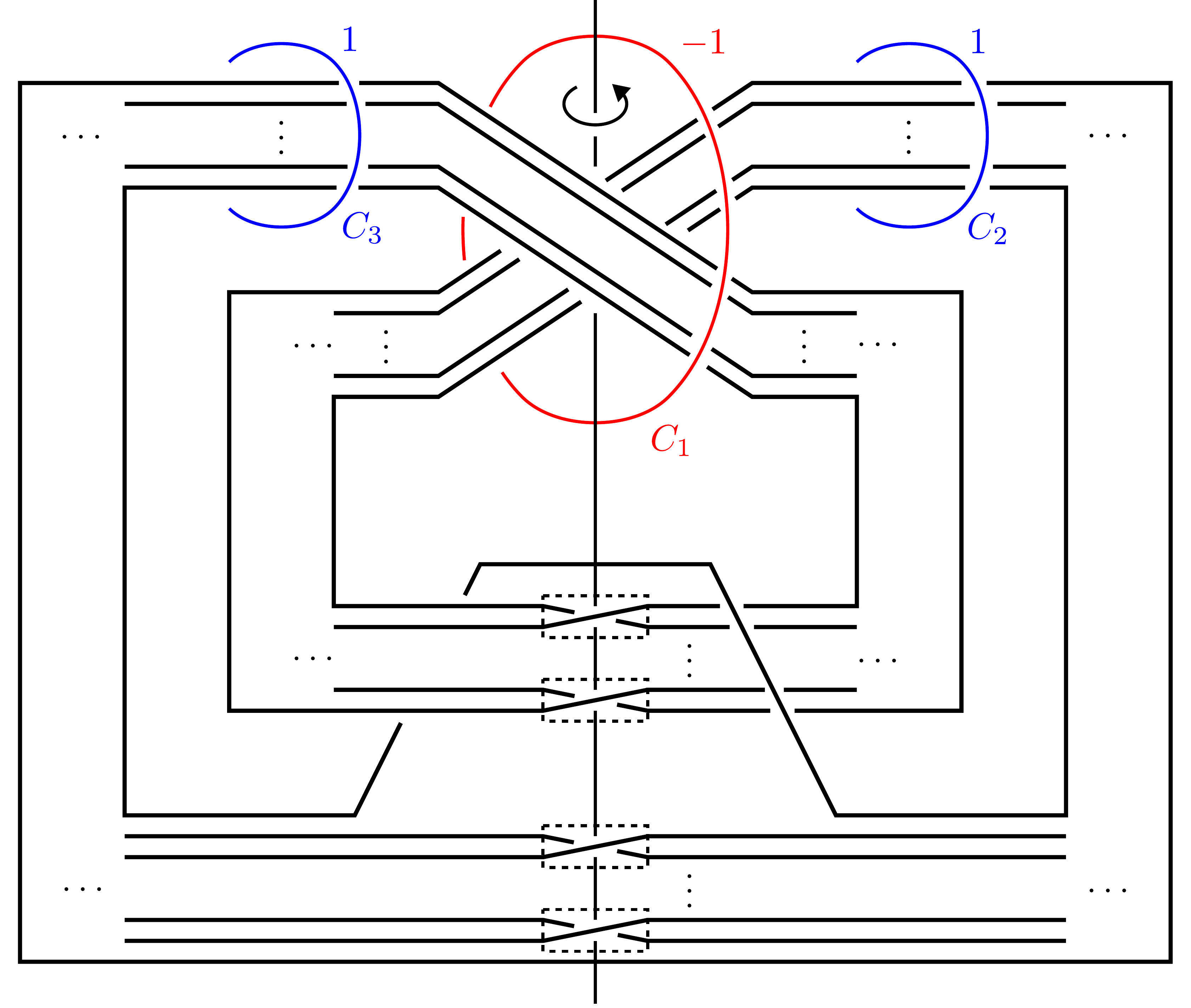}
\caption{A strongly invertible position of the link $K \cup C_1 \cup C_2 \cup C_3$. The dashed boxes total $2k-1$: $k$ in the upper part and $k-1$ in the lower part.}
\label{strongly_invertible}
\end{figure}

\begin{proposition}\label{prop_Lspace}
The $(12k^2 + 2k)$--surgery on $K_{2k}$ yields an $L$--space. 
\end{proposition}

\begin{proof}
When $k=1$, the $14$--surgery on $K_2=o9\_30634$ yields an $L$--space \cite{BK24}. Assume that $k\ge 2$.
Consider the quotient of the link in Figure \ref{strongly_invertible} under the involution around the axis, as shown in Figure \ref{quotient}. Note that the surgery coefficient on $K$ is $4k^2 + 2k$, and the writhe of $K$ in the diagram is $4k^2 + 2k + 1$. Hence, $(4k^2 + 2k)$--surgery on $K$ corresponds to a tangle replacement by the $(-1)$--tangle.

\begin{figure}
\centering
\includegraphics[scale=0.12]{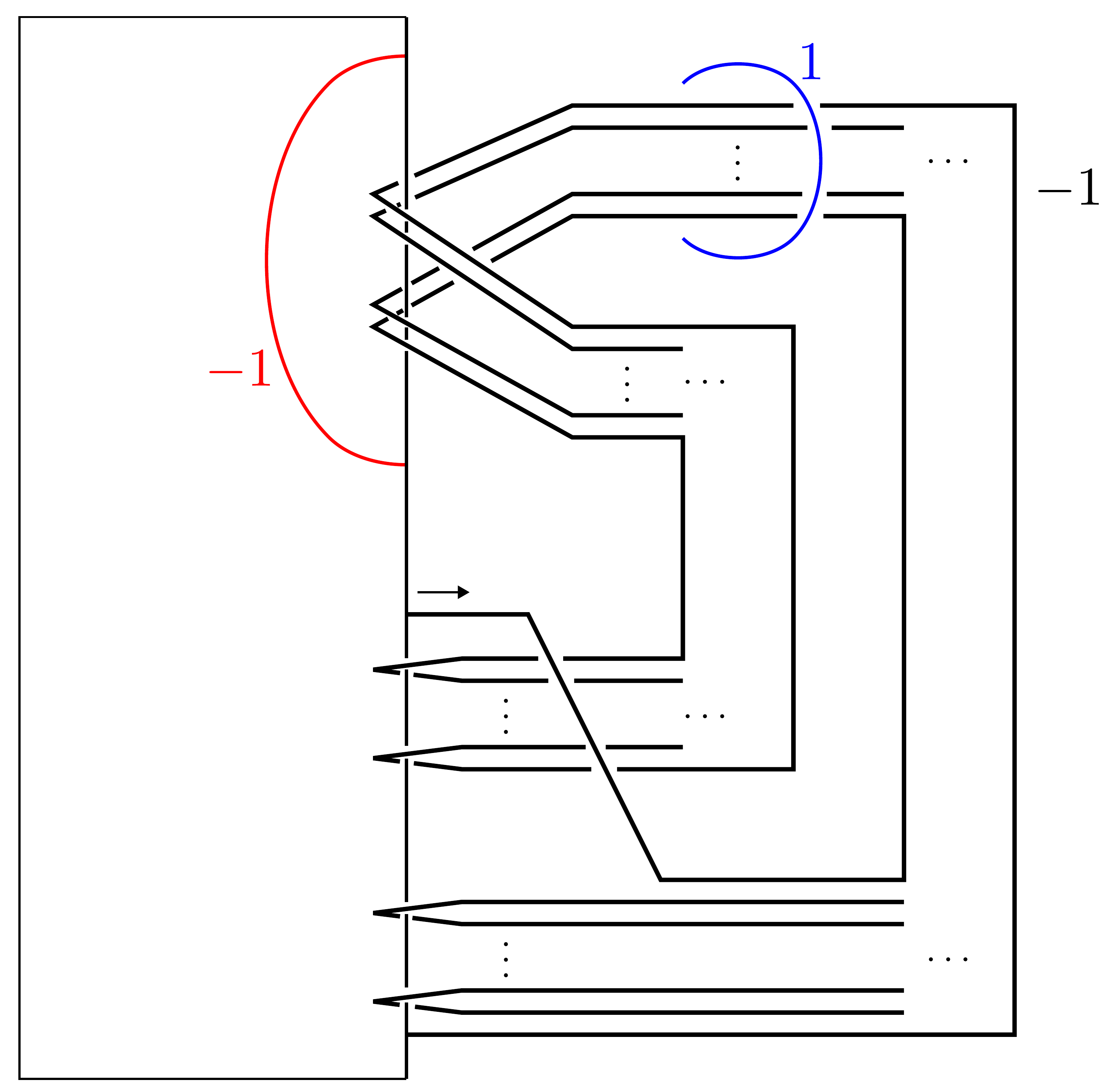}
\caption{The quotient by the involution around the axis.}
\label{quotient}
\end{figure}

Figures \ref{ell_deform_1} and \ref{ell_deform_2} illustrate a deformation of the quotient. Performing the indicated surgeries and tangle replacements on the right of Figure \ref{ell_deform_2} yields the link diagram in Figure \ref{ell_deform_3}, and we denote the resulting link by $\ell$.

\begin{figure}
\centering
\includegraphics[scale=0.10]{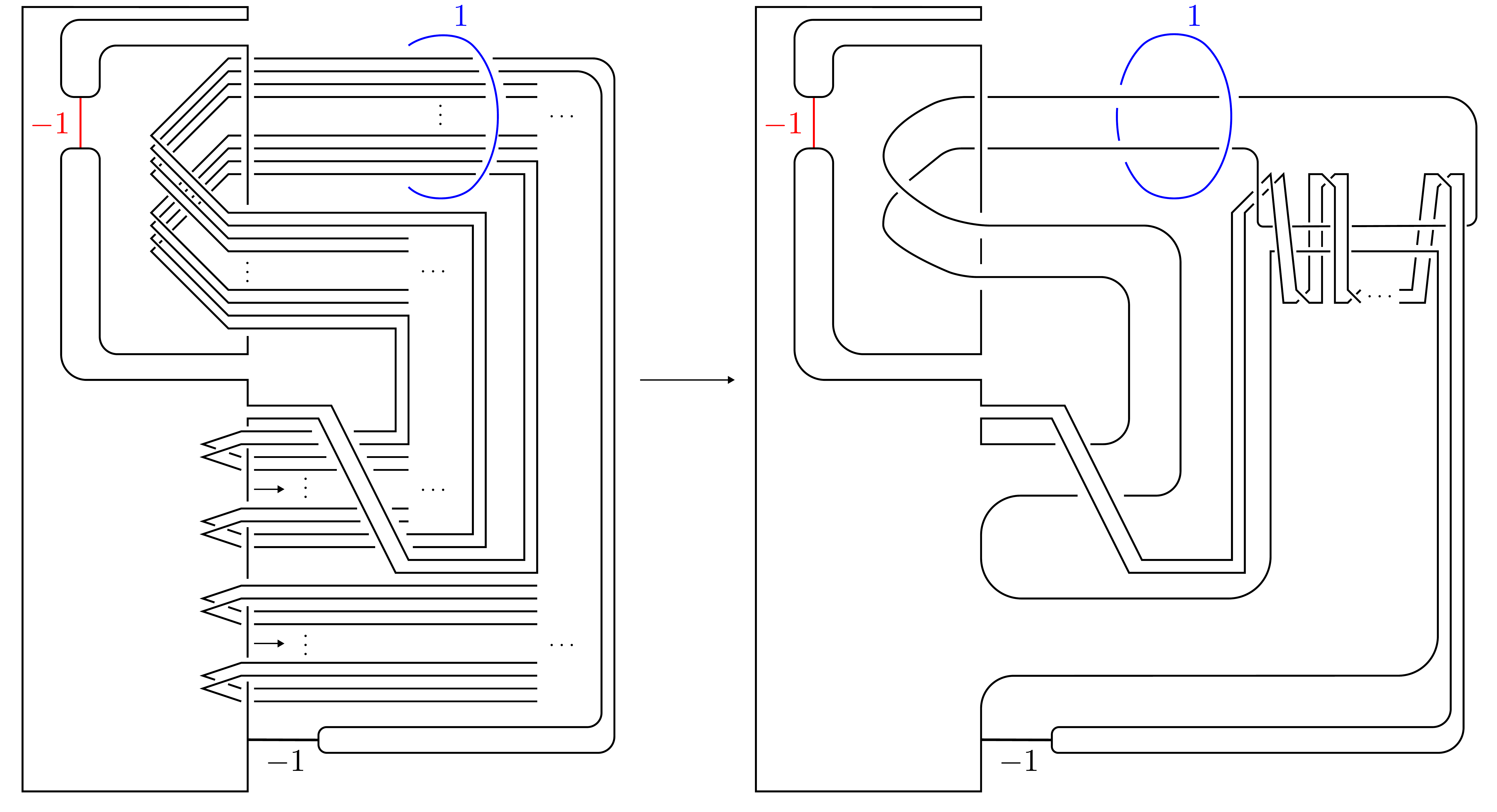}
\caption{A deformation of the quotient from Figure \ref{quotient}.}
\label{ell_deform_1}
\end{figure}

\begin{figure}
\centering
\includegraphics[scale=0.10]{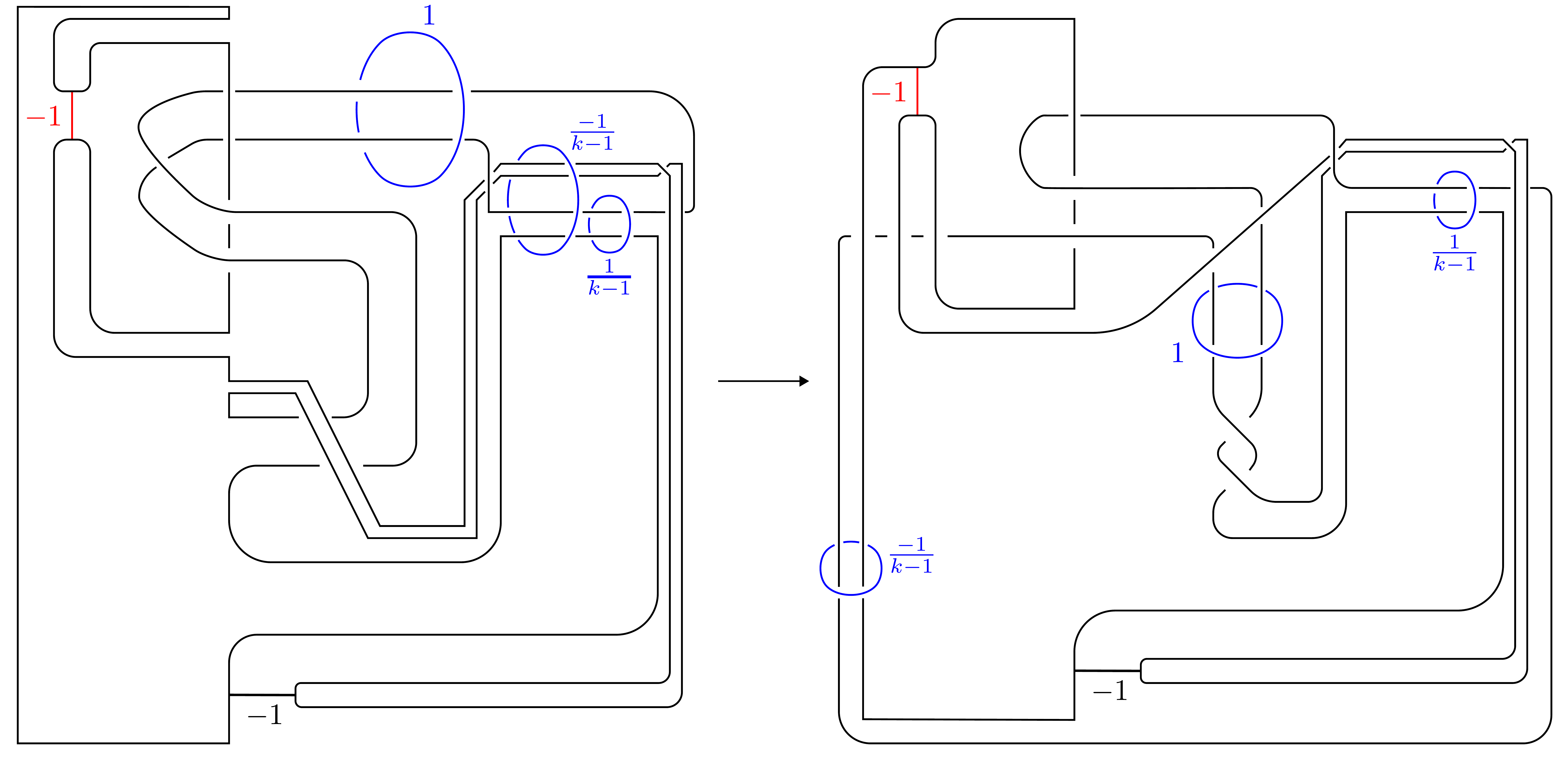}
\caption{Continued from Figure \ref{ell_deform_1}.}
\label{ell_deform_2}
\end{figure}

\begin{figure}
\centering
\includegraphics[scale=0.10]{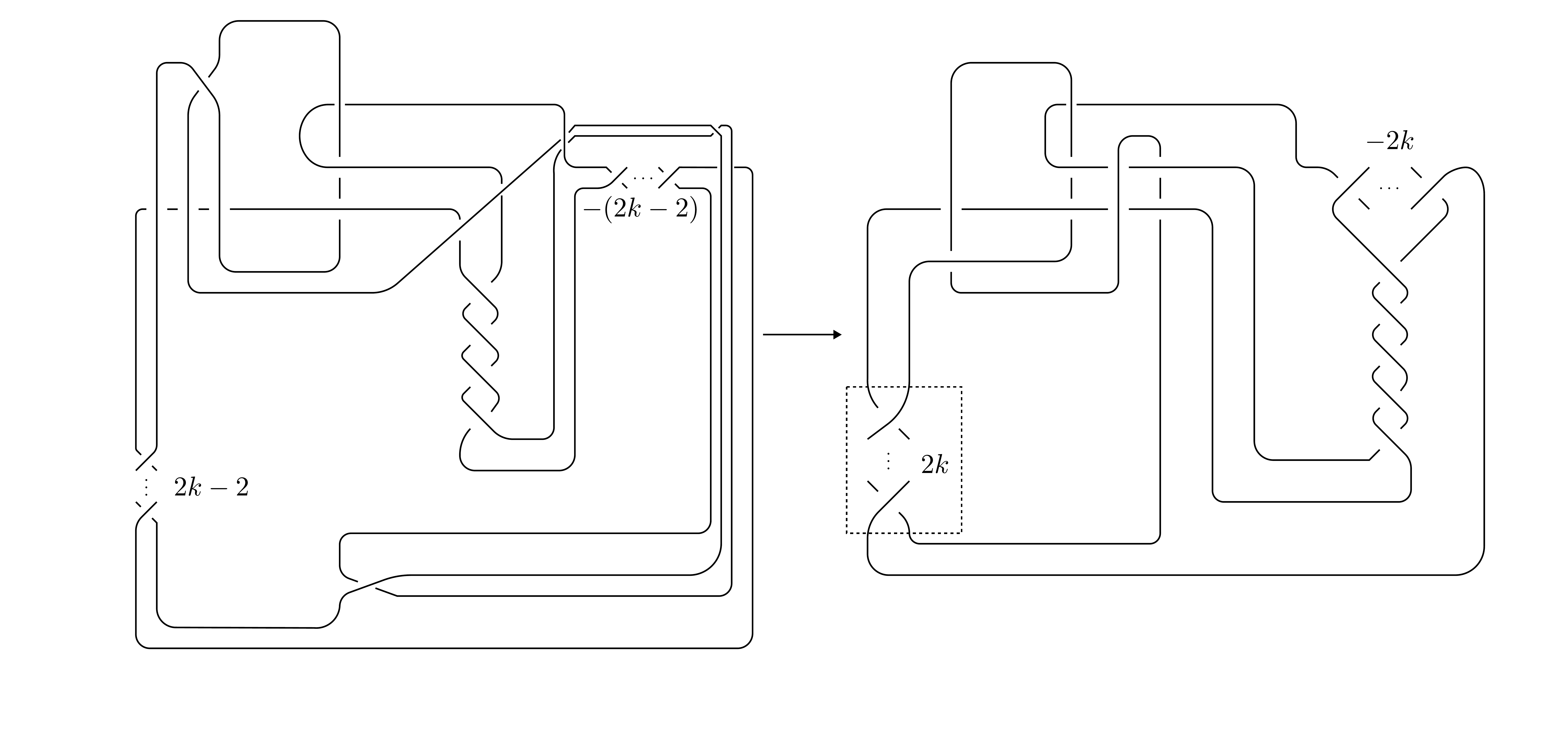}
\caption{Continued from Figure \ref{ell_deform_2}. Integers indicate the number of half-twists: right-handed if positive, left-handed otherwise.}
\label{ell_deform_3}
\end{figure}

\begin{claim}\label{DBC_Lspace}
The double branched cover of the link $\ell$ is an $L$--space.
\end{claim}

\begin{proof}
Using the Goeritz matrix derived from a checkerboard coloring of the diagram of $\ell$ (Figure \ref{ell_deform_3}), we compute $\mathrm{det}(\ell) = 12k^2 + 2k$.  
By smoothing the $2k-1$ crossings indicated by the dashed box in Figure \ref{ell_deform_3} as shown in Figure \ref{smoothing_tree}, we obtain the links (or knots) $\ell^i_{\infty}$ $(i = 1, \ldots, 2k-1)$ and a knot $\ell_0$.

A direct computation gives $\mathrm{det}(\ell^i_{\infty}) = 12k^2 + 2k - (6k + 1)i$ and $\mathrm{det}(\ell_0) = 6k + 1$. Hence,
\[
\mathrm{det}(\ell^i_{\infty}) = \mathrm{det}(\ell^{i+1}_{\infty}) + \mathrm{det}(\ell_0)
\quad \text{for } i = 1, \ldots, 2k - 2,
\quad \text{and} \quad
\mathrm{det}(\ell) = \mathrm{det}(\ell^1_{\infty}) + \mathrm{det}(\ell_0).
\]

By \cite[Proposition 2.1]{OS05A} and \cite[Proposition 2.1]{OS05B}, it suffices to show that the double branched covers of $\ell_0$ and $\ell^{2k-1}_{\infty}$ are both $L$--spaces.

\begin{figure}
\centering
\includegraphics[scale=0.12]{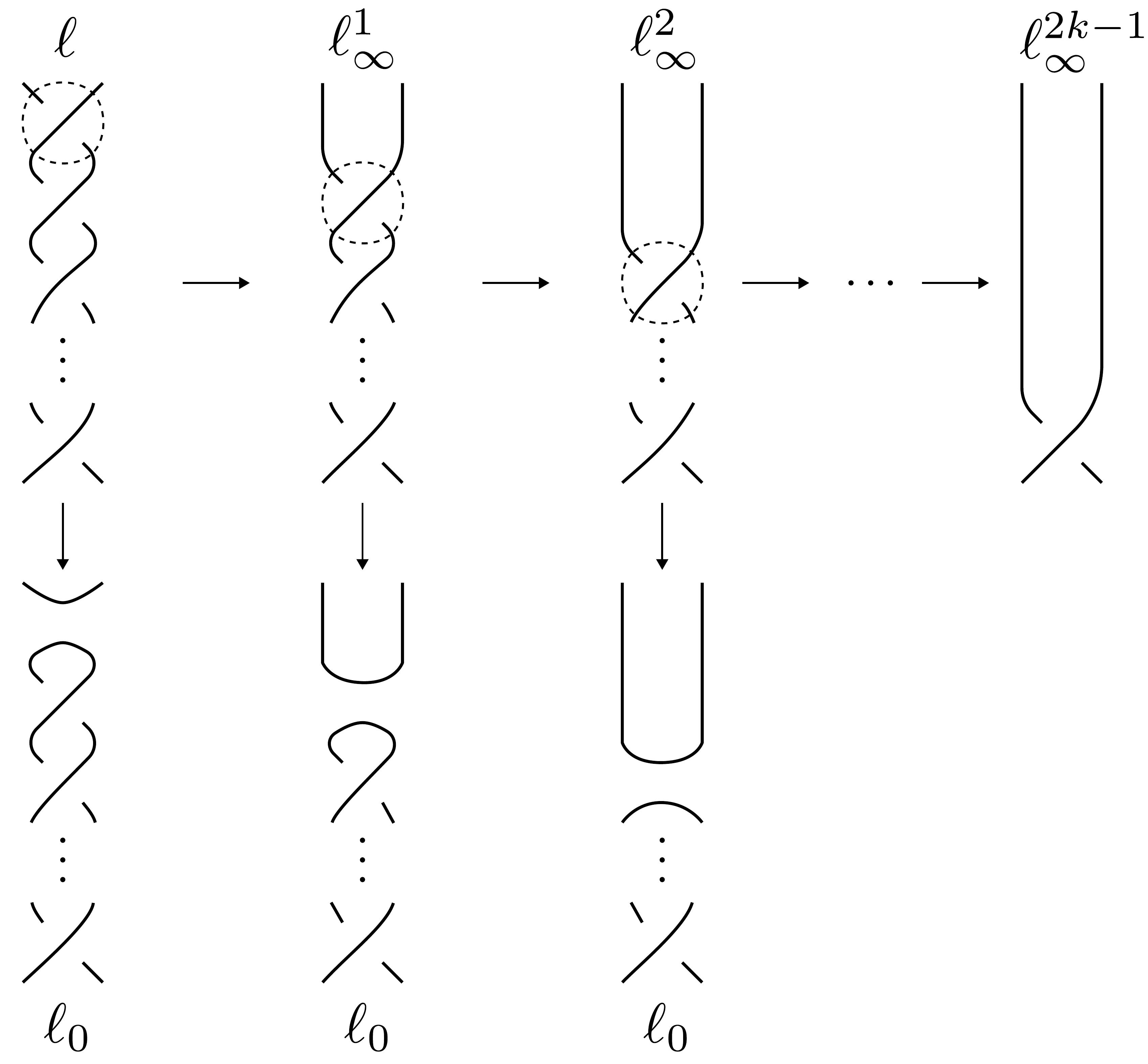}
\caption{Smoothing the crossings indicated by the dashed box in Figure \ref{ell_deform_3}.}
\label{smoothing_tree}
\end{figure}

\begin{figure}
\centering
\includegraphics[scale=0.09]{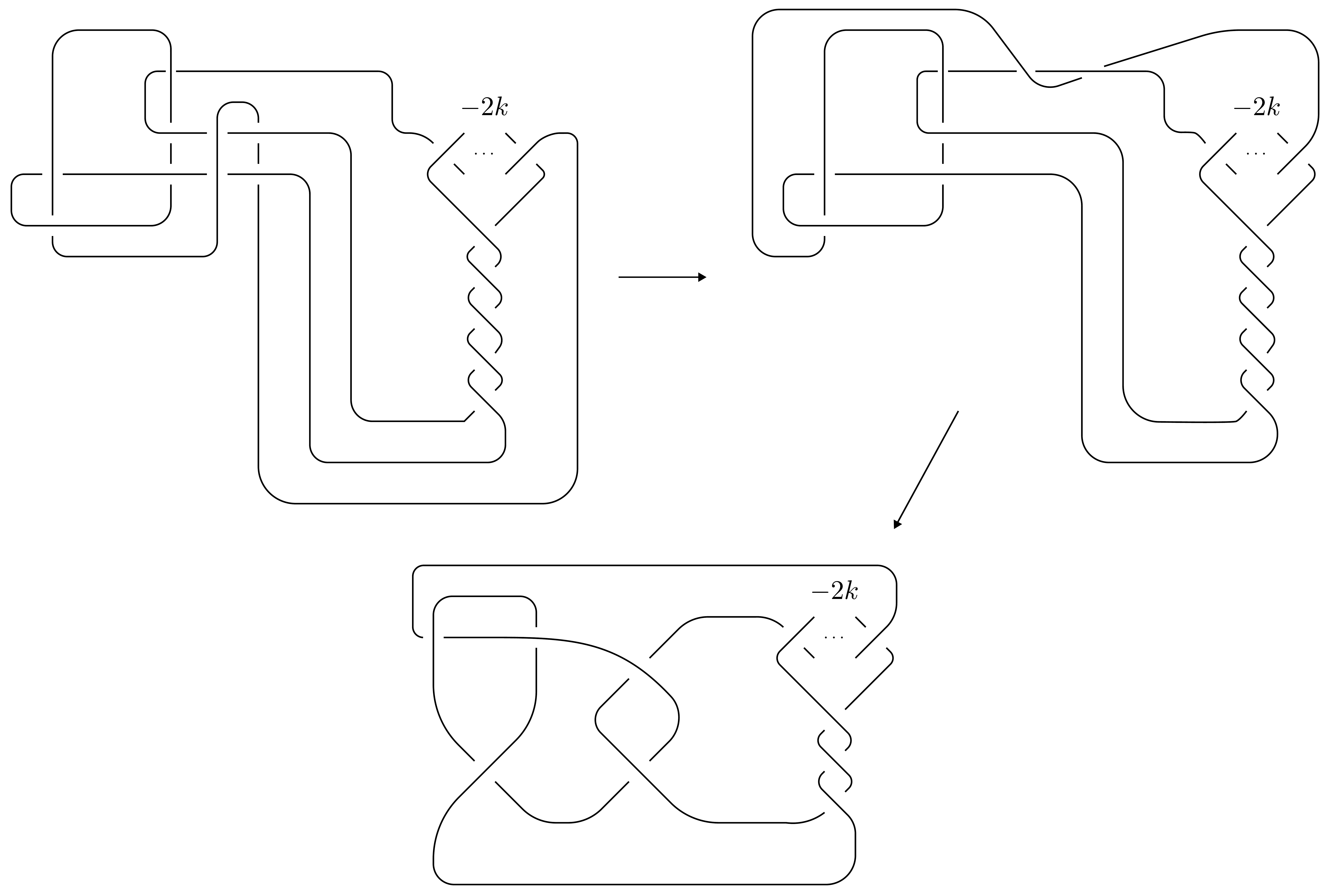}
\caption{$\ell_0$ is a Montesinos knot.}
\label{ell_0}
\end{figure}

As shown in Figure \ref{ell_0}, the knot $\ell_0$ is the Montesinos knot $M(-\frac{2}{3}, \frac{1}{2}, \frac{2k}{6k - 1})$ (not quasi-alternating by \cite{Iss18}). Its double branched cover is the Seifert fibered space
\[
M(0; -\tfrac{2}{3}, \tfrac{1}{2}, \tfrac{2k}{6k - 1}) = M(-1; \tfrac{1}{2}, \tfrac{2k}{6k - 1}, \tfrac{1}{3}),
\]
following the convention of \cite{LS07}, consistent with \cite{BK24, Ter22}.

According to \cite{LM04, LS07}, such a Seifert fibered space $M(-1; r_1, r_2, r_3)$ with $1 \ge r_1 \ge r_2 \ge r_3 \ge 0$ is an $L$--space if and only if there are no relatively prime integers $m > a > 0$ satisfying:
\[
m r_1 < a < m(1 - r_2) \quad \text{and} \quad m r_3 < 1.
\]
For $r_1 = \tfrac{1}{2}$, $r_2 = \tfrac{2k}{6k - 1}$, and $r_3 = \tfrac{1}{3}$, the condition $m r_3 < 1$ with $m > a > 0$ gives $a = 1$, $m = 2$, but this violates $m r_1 < a$. Thus, the double branched cover of $\ell_0$ is an $L$--space.

\begin{figure}
\centering
\includegraphics[scale=0.09]{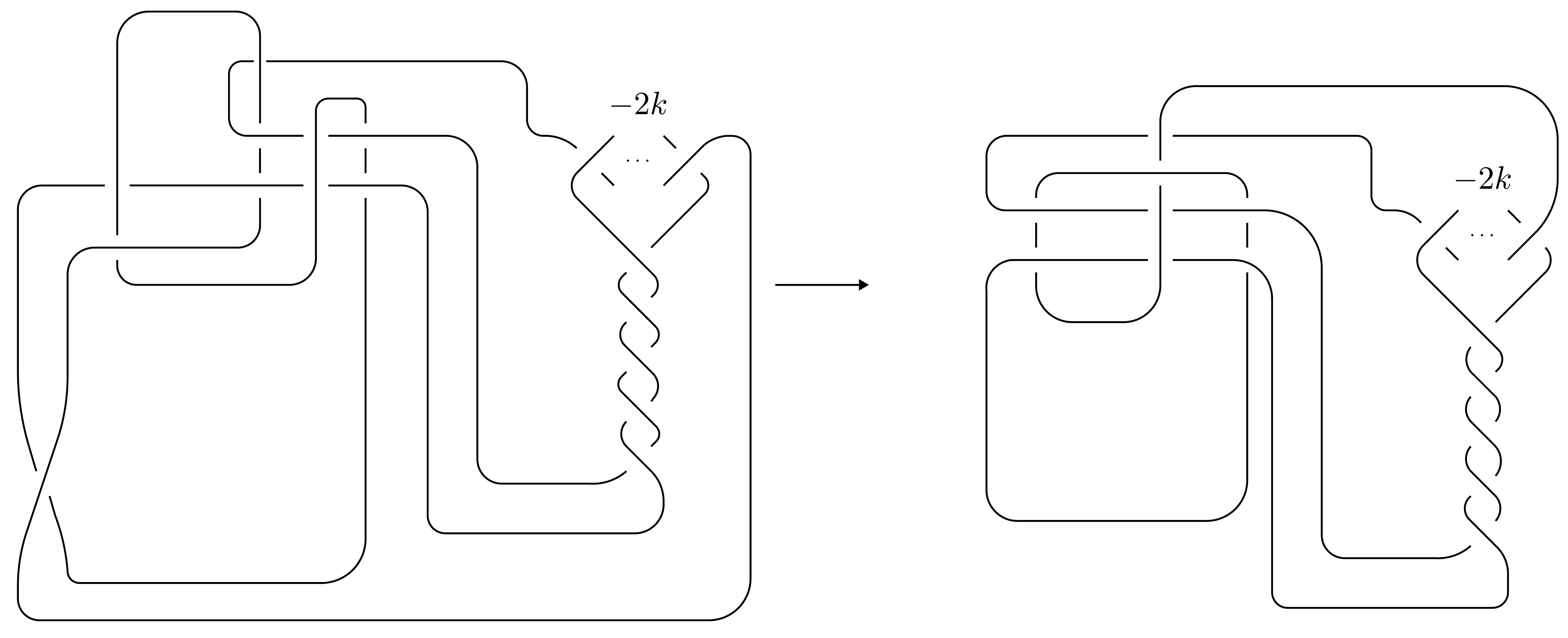}
\caption{A deformation of $\ell^{2k-1}_{\infty}$.}
\label{ell_infty_1}
\end{figure}

\begin{figure}
\centering
\includegraphics[scale=0.09]{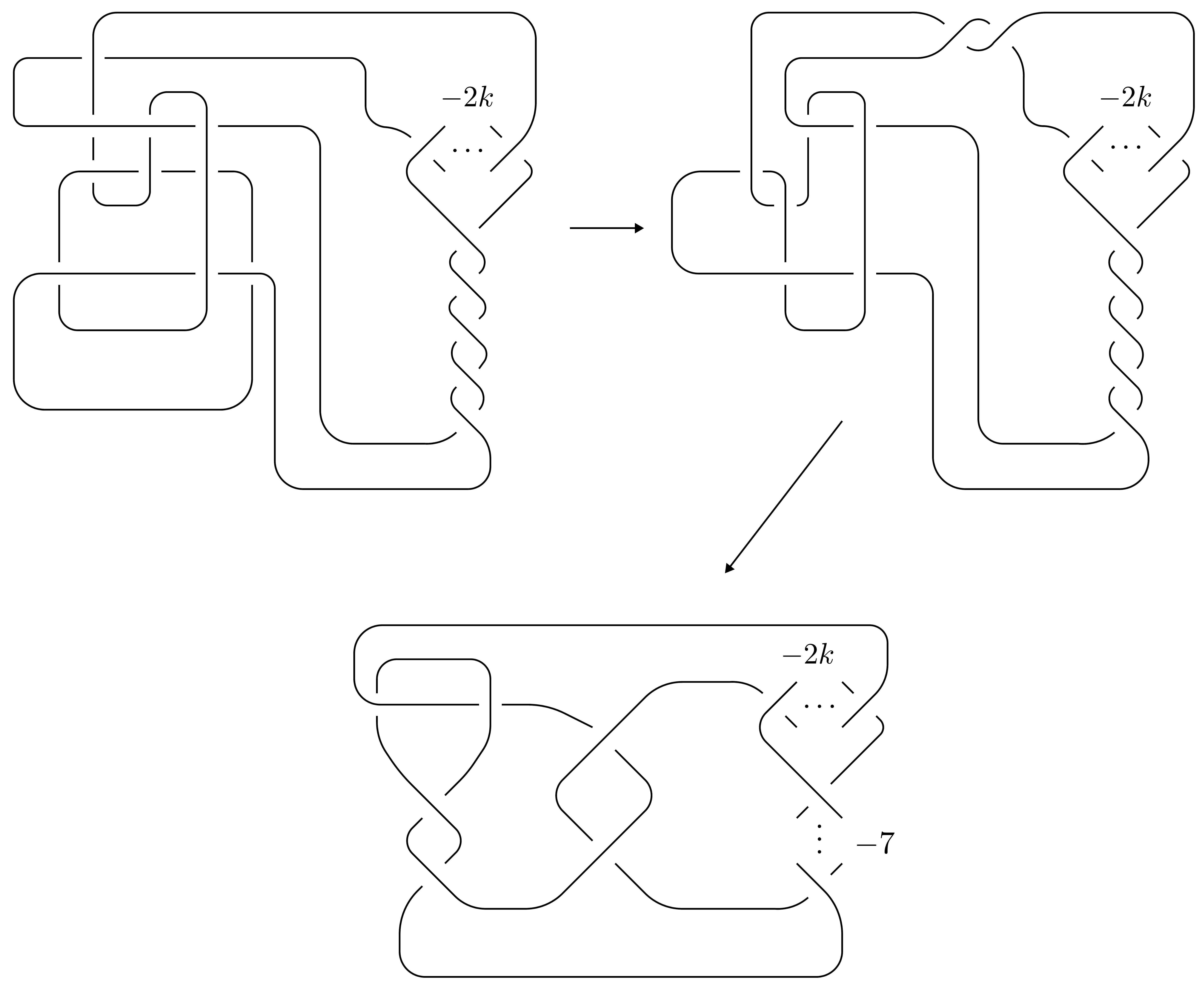}
\caption{Continued from Figure \ref{ell_infty_1}. $\ell^{2k-1}_{\infty}$ is a Montesinos knot.}
\label{ell_infty_2}
\end{figure}

Similarly, Figures \ref{ell_infty_1} and \ref{ell_infty_2} show that $\ell^{2k-1}_{\infty}$ is the Montesinos knot $M(\tfrac{2}{5}, -\tfrac{1}{2}, \tfrac{2k}{14k - 1})$, and its double branched cover is
\[
M(0; \tfrac{2}{5}, -\tfrac{1}{2}, \tfrac{2k}{14k - 1}) = M(-1; \tfrac{1}{2}, \tfrac{2}{5}, \tfrac{2k}{14k - 1}).
\]

Set $r_1 = \tfrac{1}{2}$, $r_2 = \tfrac{2}{5}$, and $r_3 = \tfrac{2k}{14k - 1}$. The condition $m r_3 < 1$ implies $m < 7 - \tfrac{1}{2k}$, so $m = 2, 3, \ldots, 6$. But for these values, there exists no integer $a$ satisfying $\tfrac{1}{2}m < a < \tfrac{3}{5}m$. Therefore, the double branched cover of $\ell^{2k-1}_{\infty}$ is an $L$--space.
\end{proof}

By the Montesinos trick and Claim \ref{DBC_Lspace}, the proof of Proposition \ref{prop_Lspace} is complete.
\end{proof}

\section{Hyperbolicity}

To complete the proof of Theorem \ref{thm_main}, we show that $K_n$ is hyperbolic. To this end, we prove that the braid $\beta_n$ representing $K_n$, regarded as an element of the mapping class group of a punctured $2$--disk, is pseudo-Anosov in the sense of the Nielsen--Thurston classification. We apply the criterion of Bestvina and Handel \cite{BH95}, see also \cite{HK06}.

Let $D$ be a $2$-disk, and $P=\{p_1,\ldots,p_k\}$ be a set of $k$ punctures. Take $k$ small circles $c_i$ $(i=1,\ldots,k)$, each centered at $p_i$, such that the interior of $c_i$  contains no other punctures. Choose a finite graph $G$ embedded in $D$ such that it is homotopy equivalent to $D \setminus P$, and contains $C = \{c_1, \ldots, c_k\}$ as a subgraph. We allow $G$ to have loops, but assume that it has no vertices of valence $1$ or $2$. Let $V(G)$ and $E(G)$ denote the sets of vertices and edges of $G$, respectively. A {\it walk\/} $\tau$ in $G$ is a finite sequence alternating between vertices and edges $(v_1,e_1,v_2,e_2,\ldots,v_{\ell},e_{\ell},v_{\ell+1})$, where $v_1,\ldots,v_{\ell+1}\in V(G)$ and $e_1,\ldots,e_{\ell}\in E(G)$ such that the endpoints of $e_i$ are $v_i$ and $v_{i+1}$. There is no confusion in denoting this walk by $\tau=e_1e_2 \cdots e_{\ell}$. Let $W(G)$ be the set of all walks in $G$.

A homeomorphism $f$ on $D$ that preserves the set of punctures $P$ induces a {\it graph map\/} 
\[
g\colon (V(G),W(G))\to (V(G),W(G)),
\]
which preserves the set $C$ set-wise. (The definition of a graph map involves the notion of a {\it fibered surface\/} associated with $G$, but we omit those details here.)

\begin{definition}
A graph map $g$ is said to be {\it efficient\/} if there are no integers $m\ge 1$ and edges $e\in E(G)$ such that 
\[
g^m(e)=\cdots e_i e_i \cdots
\]
for some $e_i\in E(G)$. If such integer $m\ge 1$ and an edge $e\in E(G)$ exist, then we say that $g^m(e)$ has a {\it back track\/}.
\end{definition}

Let $\mathrm{Vec}(G)$ be the real vector space spanned by $E(G)$. Each walk in $G$ determines an element of $\mathrm{Vec}(G)$ by mapping to the linear combination in which the coefficient of each $e_i$ is its multiplicity in the walk. For a graph map $g$, we define its transition matrix $T_g$ as the linear transformation
\[
\mathcal{T}_g\colon\textrm{Vec}(G)\to\textrm{Vec}(G)
\]
that maps each edge $e$ to $g(e)$, regarded as an element of $\mathrm{Vec}(G)$. 

Let $C^{\textrm{pre}}$ be the set of edges $e\in E(G)$ such that $g^m(e)$ is contained in $C$ for some $m\ge 1$, and define the set of {\it real edges} by 
\[
E^{\textrm{re}}(G)=E(G)\backslash (C\cup C^{\textrm{pre}}).
\]
By the definition of $g$ and $C^{\textrm{pre}}$, the transition matrix $\mathcal{T}_g$
has the block form:
\[
\mathcal{T}_g=\left(
\begin{array}{ccc}
\mathcal{C} & \mathcal{A} & \mathcal{B}\\
0 & \mathcal{C}^{\textrm{pre}} & \mathcal{D}\\
0 & 0 & \mathcal{T}_g^{\textrm{re}}
\end{array}
\right)
\]
where $\mathcal{C}$, $\mathcal{C}^{\textrm{pre}}$ and $\mathcal{T}_g^{\textrm{re}}$ are the transition matrices associated with $C$, $C^{\textrm{pre}}$ and $E^{\textrm{re}}(G)$, respectively. 

Recall that a nonnegative square matrix $M$ is called {\it irreducible\/} if, for any indices $(i,j)$, there exists  $m\ge 1$ such that the $(i,j)$--entry of $M^m$ is positive. Then, by the Perron--Frobenius theorem, such a matrix $M$ has a real, positive eigenvalue greater than the absolute values of all other eigenvalues.  This eigenvalue is called the {\it Perron--Frobenius eigenvalue\/} and is denoted by $\lambda(M)$ (which coincides with the spectral radius of M).

\begin{theorem}[\cite{BH95}]\label{thm_BH}
Let $f$ be a homeomorphism on $D$ that preserves the set of punctures $P$, and $g$ be an induced graph map for $f$. Suppose that
\begin{itemize}
\item $g$ is efficient, and
\item the transition matrix $\mathcal{T}_g^{{\rm re}}$ with respect to the real edges is irreducible with $\lambda(\mathcal{T}_g^{{\rm re}})>1$.
\end{itemize}
Then the mapping class of $f$ (up to isotopy) is pseudo-Anosov with dilatation equal to $\lambda(\mathcal{T}_g^{{\rm re}})$.
\end{theorem}

\begin{lemma}\label{lem_pA}
The mapping class corresponding to the braid 
\[
\beta_n=X_n^3\cdot [-1,-2,\ldots,-(n-1),n,n-1,n-2,\ldots,2,1,1,2,3,\ldots,n]\ (n\ge 2)
\]
is pseudo-Anosov.
\end{lemma}
\begin{proof}
When $n=2$, the braid $\beta_2$ is pseudo-Anosov by the Sage Mathematics Software System \cite{Sage}. Assume that $n\ge 3$.

Consider the braid 
\[
\beta'_n=[n,n-1,n-2,\ldots,2,1,1,2,3,\ldots,n]\cdot X_n^3\cdot [-1,-2,\ldots,-(n-1)],
\]
which is conjugate to $\beta_n$.

\begin{figure}
\centering
\includegraphics[scale=0.13]{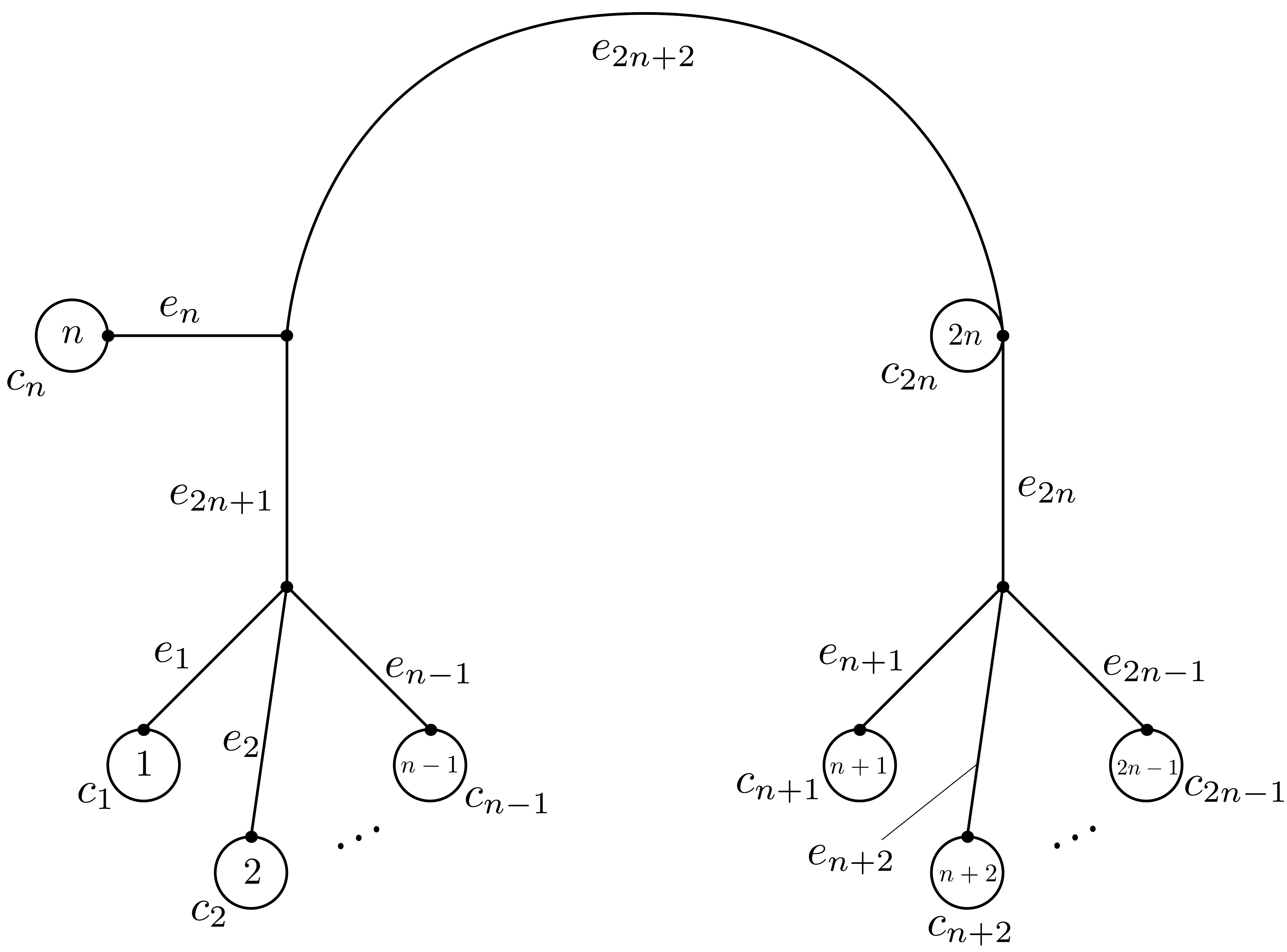}
\caption{The graph $G$ embedded in $D\backslash \{p_1,\ldots,p_{2n}\}$. Each integer inside a circle $c_i$ represents a puncture.}
\label{graph}
\end{figure}

\begin{figure}
\centering
\includegraphics[scale=0.08]{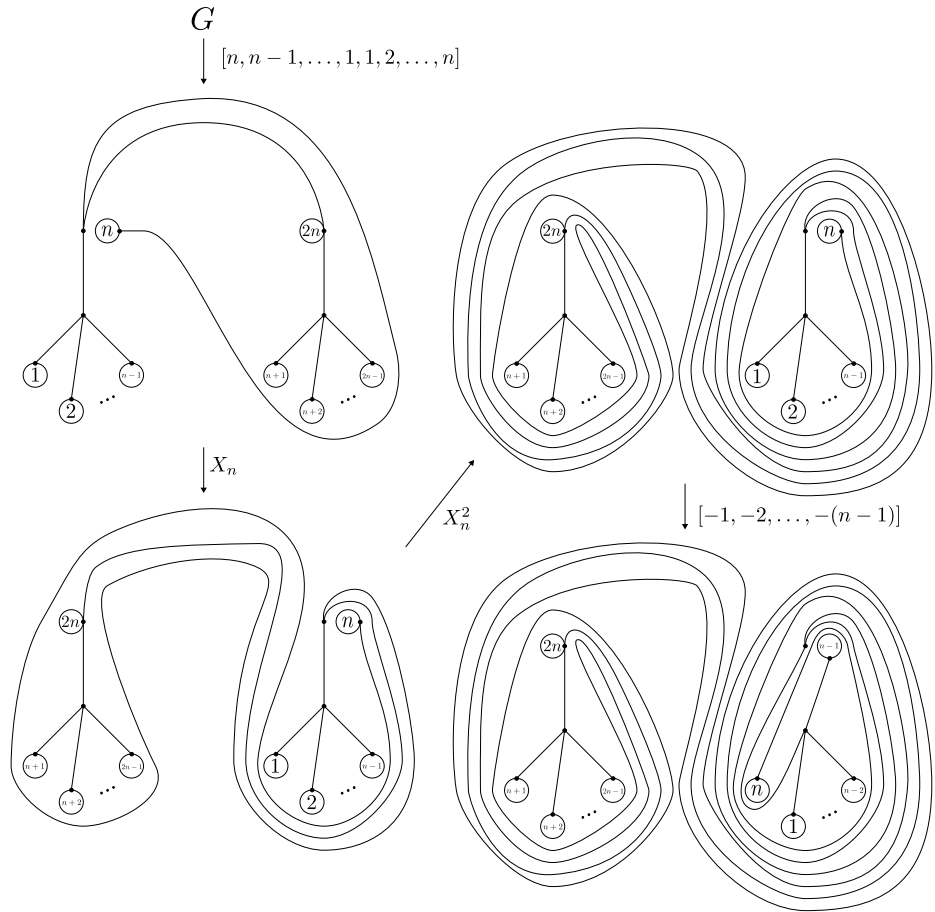}
\caption{The deformation of $G$ induced by the braid $\beta'_n$.}
\label{graph_deform_1_2}
\end{figure}

\begin{figure}
\centering
\includegraphics[scale=0.12]{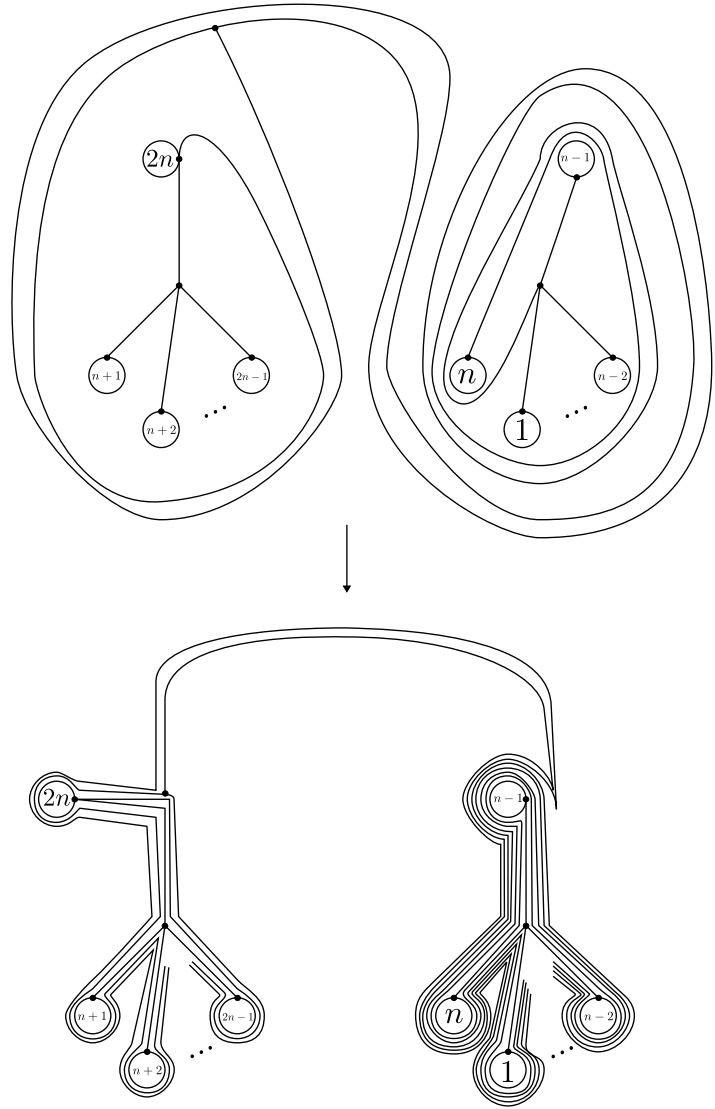}
\caption{Continued from Figure \ref{graph_deform_1_2}. Here, the deformation from the right bottom in Figure \ref{graph_deform_1_2} to the bottom of this figure is by an isotopy.}
\label{graph_deform_3}
\end{figure}

Let $G$ be a graph embedded in $D\backslash \{p_1,\ldots,p_{2n}\}$ as shown in Figure \ref{graph}. 
Figures \ref{graph_deform_1_2} and \ref{graph_deform_3} illustrate a deformation of $G$ induced by the braid $\beta'_n$. The induced graph map $g$ acts on the real edges as follows: 
\begin{align*}
g(e_i) & =e_{n+1+i}\ \text{for}\ i=1,\ldots,n-1,\\
g(e_{n+i}) & =e_i\ \text{for}\ i=1,\ldots,n-1,\\
g(e_n) & =e_{2n+1}e_{n-1}c_{n-1}\cdots e_{2n}c_{2n}e_{2n}e_{n+1},\\
g(e_{2n}) & =e_{2n+1},\\
g(e_{2n+1}) & =e_{n+1}c_{n+1}e_{n+1}\cdots e_{2n}c_{n-1}e_{2n+2},\ \text{and}\\
g(e_{2n+2}) & = e_n c_n e_n\cdots e_{n-1} e_{2n+1}e_n.
\end{align*}
Note that $e_1,\ldots,e_{2n+2}$ are real edges, and $g(e)$ has no back tracks for any real edge $e$.

Suppose that $g^k(e)$ has a back track for some $e\in E^{\textrm{re}}(G)$ and $k\ge 1$. Then it must occur at $e_n$, that is, $g^k(e)$ contains a sub-walk of the form
\[
g^k(e)=\cdots e_{2n+1} e_n e_n e_{2n+1} \cdots.
\]
This implies that $g^{k-1}(e)$ contains a sub-walk of the form 
\[
g^{k-1}(e)=\cdots e_{2n+2} e_{2n}\cdots \ \text{or}\ \cdots e_{2n}e_{2n+2}\cdots,
\]
which does not occur by Figure \ref{graph_deform_3}. Hence, $g$ is efficient.

Since $g(e_n)$ passes through all real edges of $G$, and for any real edge $e$, there exists $k\ge 1$ such that $g^k(e)$ passes through $e_n$, it follows that some power $(\mathcal{T}_g^{{\rm re}})^m$ of the transition matrix has all entries that are positive integers. Moreover, we can choose $m$ such that the trace of $(\mathcal{T}_g^{{\rm re}})^m$ is greater than $\# E^{\textrm{re}}(G)$. Therefore, $(\mathcal{T}_g^{{\rm re}})$ is irreducible and satisfies $\lambda(\mathcal{T}_g^{\textrm{re}})>1$. By Theorem \ref{thm_BH}, $\beta_n'$ is pseudo-Anosov, and hence so is $\beta_n$.
\end{proof}

\begin{proposition}\label{prop_hyperbolic}
$K_n$ $(n\ge 2)$ is a hyperbolic knot.
\end{proposition}
\begin{proof}
Note that $K_n$ can be obtained by the closure of the braid $X_n\beta_nX^{-1}_n$, and $X_n\beta_nX^{-1}_n$ is also pseudo-Anosov by Lemma \ref{lem_pA}.

\begin{figure}
\centering
\includegraphics[scale=0.08]{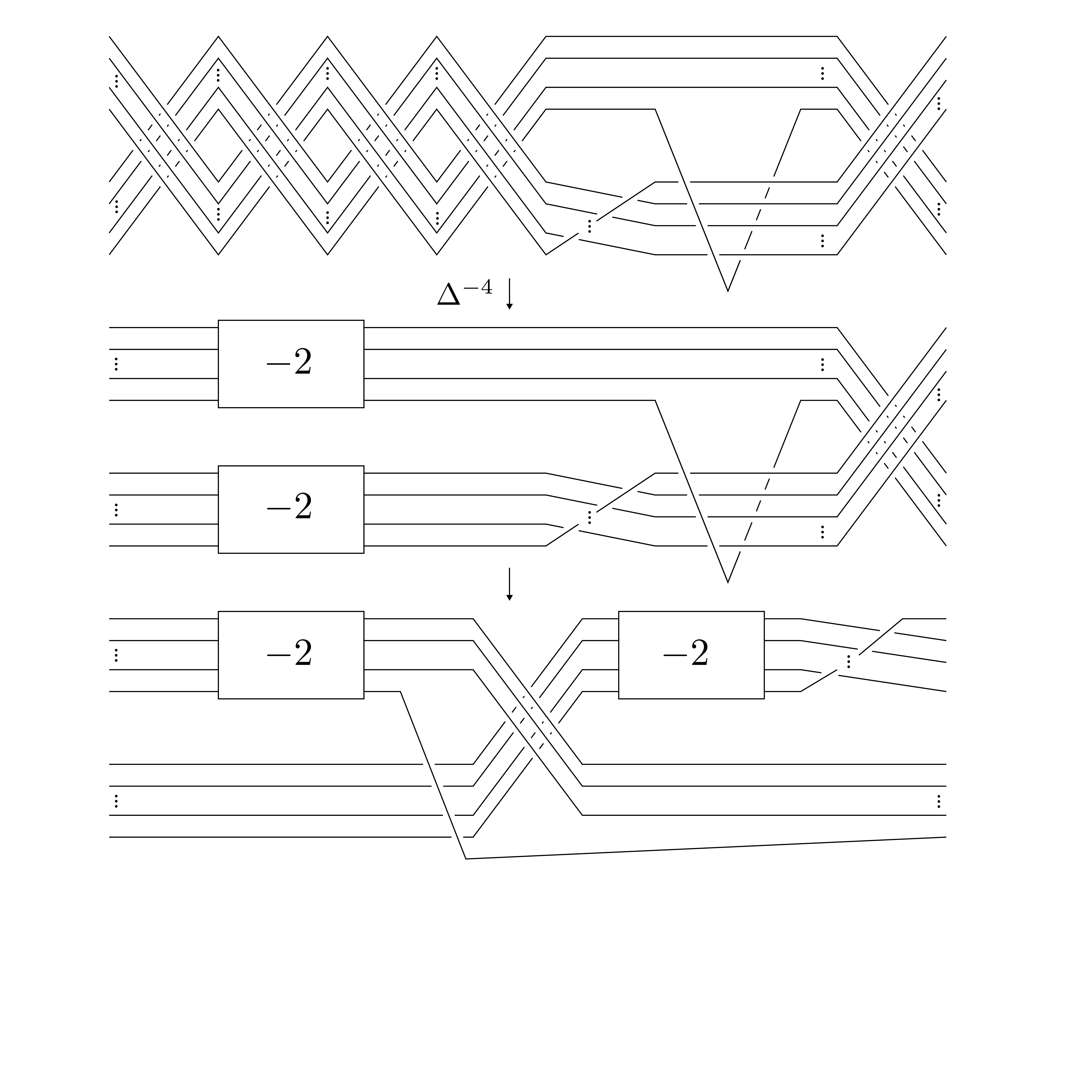}
\caption{(Top) The braid $X_n\beta_nX_n^{-1}$. (Middle) The braid obtained by rewriting the braid $\Delta^{-4}X_n\beta_nX_n^{-1}$, where $\Delta$ is the Garside fundamental $2n$--braid. Here the box with $-2$ indicates two left-handed full-twists. (Bottom) The braid obtained by further modifying the middle braid. Since this braid is $\sigma_1$--positive, we have $\Delta^4<_{\textrm{D}}X_n\beta_nX_n^{-1}$ where $<_{\textrm{D}}$ is the Dehornoy order.}
\label{K_n_Dehornoy}
\end{figure}

By Figure \ref{K_n_Dehornoy}, the Dehornoy floor of $X_n\beta_nX^{-1}_n$ is greater than $1$ (see \cite{Ito11} for definition of the Dehornoy floor).
Therefore, by Theorem 1.3 of \cite{Ito11}, the closure of $\beta_n$, namely $K_n$, is a hyperbolic knot.
\end{proof}

\begin{proof}[Proof of Theorem \ref{thm_main}]
By Propositions \ref{prop_non_braid_positive}, \ref{prop_Lspace} and \ref{prop_hyperbolic}, we have the conclusion.
\end{proof}


\end{document}